\documentclass[11pt,letterpaper,twoside,reqno,nosumlimits]{amsart}

\usepackage[usenames,dvipsnames]{xcolor}
\usepackage{fancyhdr}
\usepackage{amsmath,amsfonts,amsbsy,amsgen,amscd,amssymb,amsthm}
\usepackage{url}

\usepackage{mathtools}

\usepackage[font=small,margin=0.25in,labelfont={bf},labelsep={colon}]{caption}

\usepackage{subcaption}

\usepackage{tikz}
\usepackage{microtype}
\usepackage{enumitem}

\definecolor{dark-gray}{gray}{0.3}
\definecolor{dkgray}{rgb}{.4,.4,.4}
\definecolor{dkblue}{rgb}{0,0,.5}
\definecolor{medblue}{rgb}{0,0,.75}
\definecolor{rust}{rgb}{0.5,0.1,0.1}

\usepackage[colorlinks=true]{hyperref}

\hypersetup{urlcolor=rust}
\hypersetup{citecolor=dkblue}
\hypersetup{linkcolor=dkblue}

\usepackage{setspace}

\usepackage{graphicx}
\usepackage{booktabs,longtable,tabu} 
\usepackage{multirow} 
\usepackage{float}

\usepackage[full]{textcomp}

\usepackage[scaled=.98,sups]{XCharter}\usepackage[scaled=1.04,varqu,varl]{inconsolata}\usepackage[type1]{cabin}\usepackage[charter,vvarbb,scaled=1.07]{newtxmath}
\usepackage[cal=boondoxo]{mathalfa}
\usepackage[T1]{fontenc}

\usepackage{bm} 

\graphicspath{{figures/}}

\newtheorem{theorem}{Theorem}[section]
\newtheorem{lemma}[theorem]{Lemma}

\newtheorem{proposition}[theorem]{Proposition}
\newtheorem{fact}[theorem]{Fact}

\newtheorem{corollary}[theorem]{Corollary}

\theoremstyle{definition}

\numberwithin{equation}{section} 
\numberwithin{figure}{section}
\numberwithin{table}{section}

\floatstyle{plaintop}
\newfloat{recipe}{thp}{lor}
\floatname{recipe}{Recipe}
\numberwithin{recipe}{section}

\providecommand{\mathbold}[1]{\bm{#1}}

\renewcommand{\phi}{\varphi}

\newcommand{\suml}{\sum\limits}

\newcommand{\econst}{\mathrm{e}}

\providecommand{\mathbbm}{\mathbb} 
\newcommand{\R}{\mathbbm{R}}

\newcommand{\diff}[1]{\mathrm{d}{#1}}
\newcommand{\idiff}[1]{\, \diff{#1}}

\newcommand{\vct}[1]{\mathbold{#1}}
\newcommand{\mtx}[1]{\mathbold{#1}}

\newcommand{\triplenorm}[1]{{\left\vert\kern-0.25ex\left\vert\kern-0.25ex\left\vert #1
    \right\vert\kern-0.25ex\right\vert\kern-0.25ex\right\vert}}

\allowdisplaybreaks

\usepackage[numbers]{natbib}

\newcommand{\om}{\omega}
\renewcommand{\th}{\theta}

\mathtoolsset{showonlyrefs}

\begin{document}

\title[Exact self-similar finite-time blowup of the HL model]{Exact self-similar finite-time blowup of the Hou--Luo model\\ with smooth profiles}
\author[D. Huang, X. Qin, X. Wang, and D. Wei]{De Huang$^1$, Xiang Qin$^2$, Xiuyuan Wang$^3$, and Dongyi Wei$^4$}
\thanks{$^1$ School of Mathematical Sciences, Peking University. Email: dhuang@math.pku.edu.cn}
\thanks{$^2$ Computing and Mathematical Sciences, California Institute of Technology. Email: xqin2@caltech.edu.}
\thanks{$^3$ School of Mathematical Sciences, Peking University. Email: wangxiuyuan@stu.pku.edu.cn}
\thanks{$^4$ School of Mathematical Sciences, Peking University. Email: jnwdyi@pku.edu.cn}

\begin{abstract}
We show that the 1D Hou--Luo model on the real line admits exact self-similar finite-time blowup solutions with smooth self-similar profiles. The existence of these profiles is established via a fixed-point method that is purely analytic. We also prove that the profiles satisfy some monotonicity and convexity properties that were unknown before, and we give rigorous estimates on the algebraic decay rates of the profiles in the far field. Our result supplements the previous computer-assisted proof of self-similar finite-time blowup for the Hou--Luo model with finer characterizations of the profiles.
\end{abstract}

\maketitle

\section{Introduction}
We consider the 1D Hou--Luo (HL) model
\begin{equation}\label{eqt:HL_model}
\begin{split}
&\om_t + u\om_x = \th_x,\\
&\th_t + u\th_x = 0,\\
&u_x = \mtx{H}(\om),
\end{split}
\end{equation}
for $x\in \R$, where $\mtx{H}(\cdot)$ denotes the Hilbert transform on the real line. This model was first proposed by Luo and Hou \cite{luo2014potentially,luo2014toward} to acquire understanding of the numerically observed self-similar singularity formation of the 3D axisymmetric Euler equations on the solid boundary of an infinitely long cylinder. It also models the boundary induced singularity formation of the 2D Boussinesq equations \cite{luo2014toward,chen2022asymptotically} in the half-space $(x_1,x_2)\in\R\times\R_+$, 
\begin{equation}\label{eqt:Boussinesq}
\begin{split}
&\om_t + u_1\om_{x_1} + u_2\om_{x_2} = \th_{x_1},\\
&\th_t + u_1\th_{x_1} + u_2\th_{x_2} = 0,\\
&(u_1,u_2) = \nabla^{\perp}(-\Delta)^{-1}\om,
\end{split}
\end{equation}
for the 2D Boussinesq equations behave similarly to the 3D axisymmetric Euler equations away from the symmetry axis; see e.g. \cite{majda2002vorticity}. In fact, the boundary finite-time blowup of 3D axisymmetric Euler equations can be approximated by the boundary finite-time blowup of the 2D Boussinesq equations up to an asymptotically small perturbation \cite{luo2014toward,chen2022stable}.

Ever since the report of convincing numerical evidence of a self-similar finite-time blowup of the 3D axisymmetric Euler equations with boundary, namely the Hou--Luo scenario \cite{luo2014potentially}, vast amounts of effort have been made to try to rigorously prove the existence of such boundary singularity for the 3D Euler/2D Boussinesq equations as well as a number of simplified models including the HL model \eqref{eqt:HL_model}. We recommend the survey paper \cite{drivas2023singularity} for a comprehensive 
literature review. We only list some of the most relevant ones here. Shortly after the original work of Luo and Hou \cite{luo2014toward}, Choi et al. \cite{choi2017finite} used a functional argument to prove the finite-time blowup of the HL model \eqref{eqt:HL_model} and another 1D model known as the CKY model \cite{choi2015finite}. However, their approach was not able to capture the self-similar nature of the blowup. Years later, Chen, Hou, and Huang \cite{chen2022asymptotically} developed a novel analysis framework based on rigorous computer-assisted proofs to establish asymptotically self-similar finite-time blowup of the HL model from smooth initial data. In particular, they constructed an approximate self-similar profile using numerical computation, and then they showed by an energy argument that any solution of the HL model that is initially close to the approximate self-similar profile (up to proper rescaling) will develop finite-time singularity in an asymptotically self-similar way. Recently, Chen and Hou \cite{chen2022stable} generalized this powerful computer-assisted approach to the higher dimension and used it to prove the asymptotically self-similar finite-time blowup of the 2D Boussinesq/3D Euler equations with boundary, thus finally settling the conjecture on the Hou--Luo scenario. Remarkably, their work showed for the first time that the 3D Euler equations can develop finite-time singularity from smooth initial data, though the presence of a solid boundary is critically necessary in this scenario. Whether this can happen in the free space $\R^3$ still remains open.

As mentioned above, the asymptotically self-similar finite-time blowups of the 2D Boussinesq equations \eqref{eqt:Boussinesq} and the 1D HL model \eqref{eqt:HL_model} were both established via a computer-assisted approach. On the one hand, computer-assisted analysis can yield much sharper or tighter estimates far beyond the reach of pure analysis, which is critical in the proofs of these blowups. On the other hand, this novel proof framework relies heavily on computer assistance, which can be quite difficult to digest and to reproduce for most of the readers. Furthermore, though the existence of an exact self-similar solution (close to the approximate one constructed numerically) is also implied by the computer-assisted approach, characterizations of the solution profiles are very limited. For example, the computer-assisted proof did not tell whether the exact self-similar profiles for the 2D Boussinesq equations or the 1D HL model are smooth. Therefore, it would be helpful to develop pure analytic strategies to prove the existence of exact self-similar finite-time blowups and to provide finer characterizations of the corresponding self-similar profiles.

In this paper, we focus on the 1D HL model \eqref{eqt:HL_model}. More specifically, we look for exact self-similar solutions of \eqref{eqt:HL_model} that take the form
\begin{equation}\label{eqt:self-similar_solution}
\om(x,t) = (T-t)^{c_\om}\cdot \Omega\left(\frac{x}{(T-t)^{c_l}}\right),\quad \th(x,t) = (T-t)^{c_\th}\cdot \Theta\left(\frac{x}{(T-t)^{c_l}}\right),
\end{equation}
where $\Omega, \Theta$ are the self-similar profiles, $c_\om, c_\th, c_l$ are the scaling factors, and $T$ is the finite blowup time. This particular self-similar ansatz is due to the natural scaling property of equations \eqref{eqt:HL_model}. Plugging this ansatz into \eqref{eqt:HL_model} and balancing the equations as $t\rightarrow T$ yields $c_\om = -1$ and $c_\th = 2c_\om + c_l$. As we will see, the undetermined value $c_l$ is related to the far-field decay rates of $\Omega$ and $\Theta$. 

In contrast to an exact self-similar solution, an asymptotically self-similar finite-time blowup refers to a solution that exhibits clear self-similarity only as $t$ approaches the finite blowup time $T$:
\begin{equation}\label{eqt:asymptotic_self-similar_solution}
\tilde \om(x,t) = (T-t)^{\tilde c_\om(t)}\cdot \widetilde\Omega\left(\frac{x}{(T-t)^{\tilde c_l(t)}},t\right),\quad \tilde \th(x,t) = (T-t)^{\tilde c_\th}\cdot \widetilde\Theta\left(\frac{x}{(T-t)^{\tilde c_l}},t\right),
\end{equation}
where the profiles $\widetilde\Omega(\cdot,t), \widetilde\Theta(\cdot,t)$ and the scaling factors $\tilde c_\om(t), \tilde c_\th(t), \tilde c_l(t)$ all depend on time and will converge to some non-trivial steady state as $t\rightarrow T$. In particular, $\tilde c_\om(t)$ must converge to $-1$. In their computer-assisted framework, Chen, Hou, and Huang \cite{chen2022asymptotically} proved the existence of a family of asymptotically self-similar finite-time blowups of the form \eqref{eqt:asymptotic_self-similar_solution} by showing the nonlinear quasi-stability of the corresponding dynamic rescaling equations around an approximate steady state constructed by numerical methods. The term ``a family of'' means that the initial state of \eqref{eqt:asymptotic_self-similar_solution} can be taken arbitrarily as long as it falls in a small energy-norm ball centered at the approximate steady state (up to rescaling). Furthermore, they also used a limit argument to show that, near the approximate steady state, there lies a true stable steady state that corresponds to exact self-similar profiles, though they could not provide accurate descriptions of these profiles. Nevertheless, their result gives a very accurate estimate of the spacial scaling: $|c_l-\bar c_l| \leq 6\times 10^{-5}$, where $\bar c_l = 2.99870$ is a numerically computed constant corresponding to the approximate steady state, and the bound $6\times 10^{-5}$ results from computer-assisted estimates. It is curious that $c_l$ is extremely close to but strictly smaller than $3$.

As supplementary to the existing results obtained via computer-assisted proofs, we prove the existence of an exact self-similar finite-time blowup of the form \eqref{eqt:self-similar_solution} for the HL model \eqref{eqt:HL_model} in an alternative way that is pure analytic, and we provide more detailed characterizations of the self-similar profiles.

\begin{theorem}\label{thm:main}
The 1D Hou--Luo model \eqref{eqt:HL_model} on the real line admits an exact self-similar finite-time blowup solution of the form \eqref{eqt:self-similar_solution} with $c_\om = -1$, $c_l\in(2\,,4.53)$, $c_\th=2c_\om+c_l$, and a pair of smooth profiles $\Omega, \Theta$ that satisfy the following:
\begin{enumerate}
\item $\Omega(x)$ is odd in $x$, and $\Theta(x)$ is even in $x$.
\vspace{1mm}
\item The functions $f(x) := \Omega(x)/x$ and $m(x) := \Theta'(x)/x$ are both non-increasing in $x$ for $x\geq 0$. Moreover, $f(\sqrt{s})$ and $m(\sqrt{s})$ are both convex in $s$ for $s\geq 0$. More precisely, $f'(x),m'(x)\leq 0$ and $(f'(x)/x)',(m'(x)/x)'\geq 0$ for all $x>0$.
\vspace{1mm}
\item $\Omega$ and $\Theta$ are both infinitely smooth on $\R$. Moreover, $\Omega \in L^{\infty}\cap L^q(\R)\cap \dot{H}^p(\R)$ for any $q> c_l$ and any $p\geq 1$, and $\Theta'\in L^{\infty}\cap L^q(\R)\cap \dot{H}^p(\R)$ for any $q> c_l/2$ and any $p\geq 1$.
\vspace{1mm}
\item Both the limits $\displaystyle \lim_{x\rightarrow+\infty} x^{1/c_l}\Omega(x)$ and $\displaystyle \lim_{x\rightarrow+\infty} x^{2/c_l}\Theta'(x)$ exist and are positive and finite. 
\end{enumerate}
\end{theorem}

Let us remark on our result. The existence of the exact self-similar profiles is proved via a nonlinear fixed-point method. More precisely, we first construct two nonlinear nonlocal maps $\mtx{M}$ and $\mtx{R}$ over a suitable function $D$ such that, if $f\in \mathbb{D}$ is a fixed point of $\mtx{R}$, i.e. $f=\mtx{R}(f)$, then $\Omega(x) = xf(x)$ and $\Theta'(x) = (c_l/2)x\mtx{M}(f)(x)$ are a pair of exact self-similar profiles of the HL model \eqref{eqt:HL_model} with $c_\om, c_\th, c_l$ given explicitly in terms of integrals of $f$. We then prove the existence of a fixed point of $\mtx{R}$ using the Schauder fixed-point theorem. A key observation in our proof is that the map $\mtx{R}$ preserves the properties that $f(x)$ is non-increasing in $x$ for $x\geq 0$ and $f(\sqrt{s})$ is convex in $s$ for $s\geq0$, which will be frequently used in our arguments. Furthermore, based on the fixed-point relation $f=\mtx{R}(f)$, we are able to determine the regularity of $f,\mtx{M}(f)$ and their far-field decay rates, which then transfer to desired properties of the corresponding self-similar profiles $\Omega$ and $\Theta$.

One unsatisfying thing about our result is the crude estimate of $c_l$. Based on our fixed-point method, we can only show that $2<c_l\leq 2(k+1)/(k-1) \approx 4.5298$, where $k = 1 + \sqrt{10}/2$. This is of course much worse than the rigorous computer-assisted estimate $|c_l - 2.99870|\leq 6\times 10^{-5}$ obtained in \cite{chen2022asymptotically}. From this comparison, one sees clearly how rigorous computer-assisted estimates can outperform pure analytic estimates in providing sharp bounds.

Finally, we remark that our work does not prove the uniqueness (up to rescaling) of self-similar profiles for the HL model \eqref{eqt:HL_model}. Hence, we cannot conclude that the self-similar profiles $\Omega,\Theta$ obtained by our fixed-point method are identical (under proper rescaling) to those obtained in \cite{chen2022asymptotically} via the computer-assisted proof. Nevertheless, it is likely that the exact self-similar solution is unique provided that $c_\om = -1$ and $\Omega'(0)=1$. In fact, we can numerically compute the fixed point $f=\mtx{R}(f)$ by an iterative scheme and compare it to the one obtained numerically in \cite{chen2022asymptotically} by solving the dynamic rescaling equation, and we see that they match perfectly well up to only scheme errors.

Our proof strategy is modified from the fixed-point method developed in a recent work by the same authors \cite{huang2024self}, where we proved the existence of exact self-similar finite-time blowups of the a-parameterized family of the generalized Constantin--Lax--Majda (gCLM) equation \cite{okamoto2008generalization},
\begin{equation}\label{eqt:gCLM}
\om_t + au\om_x = u_x\om,\quad u_x=\mtx{H}(\om),
\end{equation}
for all $a\leq 1$. This equation is also a 1D model for the vorticity formulation of the 3D incompressible Euler equations, with a parameter $a$ that controls the competition between advection and vortex stretching. The case $a=0$ is the original CLM model \cite{constantin1985simple}, the case $a=1$ is also known as the De Gregorio model \cite{de1990one}, and the case $a=-1$ corresponds to the C\'ordoba--C\'ordoba--Fontelos model \cite{cordoba2005formation}. While the HL model represents a relatively recent development in this field, the gCLM model has been studied extensively for decades. We only mention below some most relevant results on finite-time singularities of the gCLM model, and we refer the reader to \cite{choi2017finite,elgindi2020effects} for a more comprehensive survey on this subject. 

For the original CLM model with $a=0$, finite-time singularity was established simultaneously with the proposal of the model in \cite{constantin1985simple} by solving the equation explicitly with suitable initial data. It was only recently that Elgindi and Jeong \cite{elgindi2020effects} discovered an exact self-similar finite-time blowup for the CLM model. Based on this exact self-similar solution, they also proved in the same paper the existence of self-similar finite-time blowups from smooth initial data for $|a|$ small enough using a series expansion argument. Later, Elgindi, Ghoul, and Masmoudi \cite{elgindi2021stable} improved on this result by establishing the stability of those self-similar blowups for sufficiently small $|a|$. A similar result was also established independently in a work of Chen, Hou, and Huang \cite{chen2021finite}. In a similar spirit, Lushnikov, Silantyev, Siegel \cite{lushnikov2021collapse} and Chen \cite{chen2020singularity} independently found an exact self-similar solution for $a=1/2$. Chen also proved stable self-similar finite-time singularities from smooth initial data for $a$ close to $1/2$ using the method developed in \cite{chen2021finite}. Finite-time singularity in the case $a=1$ was conjectured by De Gregorio \cite{de1990one} and was first rigorously established in a self-similar form by Chen, Hou, and Huang \cite{chen2021finite} using a computer-assisted proof (the same proof framework as in \cite{chen2022asymptotically}). Later, Huang, Tong, and Wei \cite{huang2023self} further showed that the De Gregorio model actually admits infinitely many self-similar finite-time blowup solutions of the same blowup scaling but with distinct profiles (under re-scaling) that all correspond to the eigen-functions of a self-adjoint, compact operator. All these self-similar finite-time blowup results above concern only the cases of $a=0,1/2,1$ and $a$ very close to these three values, where the (approximate) self-similar profiles were either found explicitly or computed numerically. Moreover, self-similar blowups with $C^\alpha$ profiles have also been constructed \cite{elgindi2020effects,zheng2023exactly} for a continuous range of $a$ under the constraint $|a|\alpha\ll1$. Finally, Huang, Qin, Wang, and Wei \cite{huang2024self} proved the existence of self-similar finite-time blowups with (interiorly) smooth profiles of the gCLM model for all $a\leq1$ via a novel unified fixed-point approach.

The main difficulty of modifying our fixed-point method from the gCLM case to the HL case lies in that the gCLM model is an equation of one active scalar $\om$, while the HL model is a coupled system of two scalars $\om, \th$. Critically and surprisingly, though the nonlinear map $\mtx{R}$ for the HL system is formally much more complicated than the one for the gCLM equation, it still enjoys the crucial property that it preserves the previously mentioned monotonicity and convexity from $f$ to $\mtx{R}(f)$. There are two major reasons why this could work. Firstly, from a more essential perspective, this nice property of $\mtx{R}$ owes to the kernel structure of the Hilbert transform $\mtx{H}$, which has already been exploited in \cite{huang2024self}. Secondly, the coupling between $\om$ and $\th$ is not too complicated in the sense that the $\th$ equation in \eqref{eqt:HL_model} does not involve $\om$ explicitly. Later, we will see that this fact enables us to construct the map $\mtx{R}$ hierarchically with some intermediate maps (including the map $\mtx{M}$ that is new in this work). It is worth mentioning that such coupling structure also appears in the CKY model \cite{choi2015finite,choi2017finite}, which takes the same form as in \eqref{eqt:HL_model} except that the relation $u_x=\mtx{H}(\om)$ is replaced with $(u/x)_x = \om/x$. Therefore, it might also be possible to apply our fixed-point method to the CKY model and prove the existence of a self-similar finite-time blowup. Moreover, we believe that this fixed-point framework can be further generalized to prove the existence of an exact self-similar finite-time blowup of the 2D Boussinesq system \eqref{eqt:Boussinesq} that has a similar coupling structure between $\om$ and $\th$. This shall be our next step in this line of research.

The remainder of this paper is organized as follows. In Section \ref{sec:self-similar_equations}, we derive equations for the self-similar profiles and then transform them into an equivalent fixed-point formulation. Section \ref{sec:fixed-point_method} is devoted to proving the existence of exact self-similar profiles via a fixed-point method, and Section \ref{sec:properties} is devoted to the establishment of the claimed properties. Finally, we perform some numerical simulations in Section \ref{sec:numerical} based on the fixed-point method to verify and visualize our theoretical results.

\section{Equations for the self-similar profiles}\label{sec:self-similar_equations}
Assuming that \eqref{eqt:self-similar_solution} is an exact self-similar solution of the HL model \eqref{eqt:HL_model}, we first derive a nonlocal ordinary differential system for the self-similar profiles $\Omega,\Theta$ and the scaling factors $c_l,c_\om,c_\th$. Under some natural regularity conditions on $\Omega$ and $\Theta$, we perform a change of variables and then transform the self-similar equations into a fixed-point formulation in the new variables.

\subsection{Self-similar profiles} Substituting the self-similar ansatz \eqref{eqt:self-similar_solution} into the equation \eqref{eqt:HL_model} yields
\begin{equation*}
\begin{split}
&-(T-t)^{c_\om-1}c_\om\Omega + (T-t)^{c_\om-1}c_lX\Omega_X + (T-t)^{2c_\om}U\Omega_X = (T-t)^{c_\th-c_l}\Theta_X,\\
&-(T-t)^{c_\th-1}c_\th\Theta + (T-t)^{c_\th-1}c_lX\Theta_X + (T-t)^{c_\th+ c_\om}U\Theta_X = 0,\\
&U_X = \mtx{H}(\Omega),
\end{split}
\end{equation*}
where $X = x/(T-t)^{c_l}$. Balancing the above equations as $t\rightarrow T$ yields $c_\om = -1$, $c_\th = c_l + 2c_\om$, and a system for the self-similar profiles:
\begin{equation*}
\begin{split}
&(c_lX+U)\Omega_X = c_\om\Omega + \Theta_X,\\
&(c_lX+U)\Theta_X = c_\th\Theta,\\
&U_X = \mtx{H}(\Omega),
\end{split}
\end{equation*}

From now on, for notation simplicity, we will still use $\om,\th,u,x$ for $\Omega, \Theta, U, X$, respectively. Moreover, it is more convenient to work with the variable $v = \th_x$ instead of $\th$. We thus obtain our main equations:
\begin{equation}\label{eqt:main}
\begin{split}
&(c_lx+u)\om_x = c_\om\om + v,\\
&(c_lx+u)v_x = (2c_\om-u_x)v,\\
&u_x = \mtx{H}(\om).
\end{split}
\end{equation}
We have substituted $c_\th = c_l + 2c_\om$, and we will always do so in what follows. The expressions of $u$ and $u_x$ in terms of $\om$ are, respectively, 
\begin{equation}\label{eqt:u_formula}
\begin{split}
&u(x) = -(-\Delta)^{-1/2}\om(x) = \frac{1}{\pi}\int_{\R}\om(y)\ln|x-y|\idiff y,\\
&u_x(x) = \mtx{H}(\om)(x) = \frac{1}{\pi}P.V.\int_{\R}\frac{\om(y)}{x-y}\idiff y.
\end{split}
\end{equation}
Here, $\mtx{H}$ is the Hilbert transform on the real line with $P.V.$ denoting the Cauchy principal value.

It is worth mentioning that Chen, Hou, and Huang \cite{chen2022asymptotically} studied the dynamic rescaling equations of the HL model (see \cite[Equation (2.1)]{chen2022asymptotically}), 
\begin{equation}\label{eqt:dynamic_rescaling}
\begin{split}
&\om_t + (c_l(t)x+u)\om_x = c_\om(t)\om + v,\\
&v_t + (c_l(t)x+u)v_x = (2c_\om(t)-u_x)v,\\
&u_x = \mtx{H}(\om),
\end{split}
\end{equation} 
which is in fact equivalent to the original HL model \eqref{eqt:HL_model} under some time-dependent change of variables. They needed to choose how $c_l(t),c_\om(t)$ depend on the solution $\om(x,t),v(x,t)$ (see \eqref{eqt:non-degeneracy_condition} below) in order to capture the intrinsic blowup scaling. In this way, they reformulated the finite-time blowup problem into a stability problem. With the help of a numerically constructed approximate steady state, they showed that the solution of \eqref{eqt:dynamic_rescaling} converges to some true steady state near the approximate one, which then implies the self-similar finite-time blowup of the original HL model.

Different from their dynamic approach, our strategy is to directly find a nontrivial solution $(\om, v, c_l, c_\om)$ of equations $\eqref{eqt:main}$, which is then an exact steady state of the dynamic rescaling equations \eqref{eqt:dynamic_rescaling}. One important fact is that, if $(\om, v, c_l, c_\om)$ is a solution of \eqref{eqt:main}, then
\begin{equation}\label{eqt:scaling}
(\om_{\alpha,\beta}(x),v_{\alpha,\beta}(x),c_{l,\alpha},c_{\om,\alpha}) = (\alpha\om(\beta x),\alpha^2v(\beta x),\alpha c_l, \alpha c_\om)
\end{equation}
is also a solution of \eqref{eqt:main} for any $\alpha \in \R, \beta> 0$. Owing to this scaling property, we can release the restriction that $c_\om=-1$. In fact, it is the ratio $c_l/c_\om$ that matters. Moreover, in consistence with the previous computer-assisted work \cite{chen2022asymptotically}, we look for solutions to the profile equations \eqref{eqt:main} that satisfy the following assumptions:
\begin{itemize}
\item Odd symmetry: $\om$ and $v$ are both odd functions of $x$. As a consequence, $u$ is also odd in $x$.
\item Non-degeneracy: $\om'(0)> 0$ and $v'(0)> 0$.
\item Out-pushing condition: $c_l + u/x > 0$ for all $x\in \R$.
\end{itemize}

These conditions are all satisfied by the approximate steady state constructed in \cite{chen2022asymptotically}, and they are actually critical to the stability argument there. The non-degeneracy condition not only ensures the solution is not trivial, but also determines how the scaling factors are related to the profiles:
\begin{equation}\label{eqt:non-degeneracy_condition}
c_l = 2\frac{v'(0)}{\om'(0)}, \quad c_\om = \frac{1}{2}c_l + u'(0).
\end{equation}
To derive these relations, one simply evaluates the derivatives of the first two equations in \eqref{eqt:main} at $x=0$ and uses the non-degeneracy condition and the odd symmetry. Conversely, imposing the relations \eqref{eqt:non-degeneracy_condition} on the dynamic rescaling equations \eqref{eqt:dynamic_rescaling}, as implemented in \cite{chen2022asymptotically}, ensures that $\om'(0,t)$ and $v'(0,t)$ are conserved quantities over time, i.e. $\om'(0,t) = \om'(0,0)$, $v'(0,t) = v'(0,0)$, which is essentially the source of stability. The odd symmetry, which is preserved by the dynamic rescaling equations, then ensures that the origin $x=0$ is a stable stagnation point that ``generates'' stability. Finally, the out-pushing condition implies that the velocity $c_lx + u$ has the same sign as $x$, and thus it transports the stability from the origin towards the far-field. 

Finally, we remark that it remains open whether there exist nontrivial odd solutions of equations \eqref{eqt:main} that are more degenerate at $x=0$ in the sense that $\om'(0) = v'(0) = 0$ while $\om^{(n)}(0)\neq 0, v^{(n)}(0)\neq 0$ for some odd integer $n>1$. Nevertheless, numerical evidence of more degenerate self-similar finite-time blowup solutions of the HL model has been reported in the thesis of Liu \cite{liu2017spatial}. It is also worth noting that a novel multi-scale self-similar finite-time blowup phenomenon has recently been discovered and proved by Huang, Qin, and Wang \cite{luo2014toward} for the CLM model (i.e. \eqref{eqt:gCLM} with $a=0$) for more degenerate initial data, which is qualitatively different from the classical one-scale self-similar blowup constructed in \cite{elgindi2020effects} with the non-degeneracy condition ($\om'(0)\neq 0$).

\subsection{Reformulation of the problem} Now, we move on to transforming equations \eqref{eqt:main} into a fixed-point formulation. In view of \eqref{eqt:scaling}, we may assume that $\om'(0)=1$, in which case $c_l = 2v'(0)$. This uses one degree of freedom in the scaling property \eqref{eqt:scaling}. Since $\om,v,u$ are assumed to be odd functions of $x$, we consider the change of variables:
\begin{equation}\label{eqt:change_variable_1}
f := \frac{\om}{x},\quad m := \frac{2}{c_l}\cdot\frac{v}{x},\quad g := \frac{c_l + u/x}{c_l + u'(0)}.
\end{equation}
Note that $f(0)=m(0) = g(0)=1$, and that the out-pushing condition implies $g(x)>0$ for all $x$. Moreover, we define
\begin{equation}\label{eqt:change_variable_2}
b := -u'(0) = -\mtx{H}(xf)(0),\quad c := c_l + u'(0) = c_l - b,\quad d := \frac{c_l}{2(c_l+u'(0))} = \frac{b+c}{2c}.
\end{equation}
Substituting these changes of variables into \eqref{eqt:main} yields
\begin{equation}\label{eqt:f-g-m}
\begin{split}
&xgf' = (1-g)f - df + dm,\\
&xgm' = 2(1-g)m - xg'm,\\
&g = 1 - \frac{(-\Delta)^{-1/2}(xf)/x-b}{c}.
\end{split}
\end{equation}
Note that the second equation of \eqref{eqt:f-g-m} does not involve $f$ explicitly. We then easily get
\[m(x) = \frac{1}{g(x)}\cdot \exp\left(2\int_0^x\frac{1-g(y)}{yg(y)}\idiff y\right),\]
meaning that $m$ can be explicitly determined by $g$. Next, we rearrange the first equation of \eqref{eqt:f-g-m} to get
\[f'(x) + d\,\frac{f(x)}{x} + \left(d-1\right)\frac{1-g(x)}{xg(x)}\cdot f(x) = d\,\frac{m(x)}{xg(x)}.\]
Multiplying both sides of this equation by $x^d$ yields 
\[(x^df(x))' + \left(d-1\right)\frac{1-g(x)}{xg(x)}\cdot x^df(x) = d\,x^{d-1}\frac{m(x)}{g(x)},\]
which leads to 
\[
f(x) = x^{-d}\exp\left((d-1)\int_0^x\frac{g(y)-1}{yg(y)}\idiff y\right)\cdot \int_0^xd\, y^{d-1}\frac{m(y)}{g(y)}\exp\left((d-1)\int_0^y\frac{1-g(z)}{zg(z)}\idiff z\right)\idiff y.
\]
Remember that, besides the normalization condition $\om'(0)=1$, we still have one degree of freedom in the scaling property \eqref{eqt:scaling} at our disposal. Later, we will use it to determine the value of $c$ in \eqref{eqt:change_variable_2} as a particular functional of $f$:
\[c_l + u'(0) = c = c(f).\]
In summary, we have obtained the following fixed-point formulation of \eqref{eqt:main}:
\begin{equation}\label{eqt:fixed-point_formulation}
\begin{split}
f\quad \longrightarrow \quad &b = -\mtx{H}(xf)(0),\quad c = c(f),\quad d = \frac{b+c}{2c}, \\
\longrightarrow \quad &g = 1 - \frac{(-\Delta)^{-1/2}(xf)/x-b}{c},\\
\longrightarrow \quad &m = \frac{1}{g}\cdot \exp\left(2\int_0^x\frac{1-g}{yg}\idiff y\right),\\
\longrightarrow \quad &f = x^{-d}\exp\left((d-1)\int_0^x\frac{g-1}{yg}\idiff y\right)\cdot \int_0^xd\, y^{d-1}\frac{m}{g}\exp\left((d-1)\int_0^y\frac{1-g}{zg}\idiff z\right)\idiff y.
\end{split}
\end{equation}
From the derivation above, it is apparent that a solution of the integral equations \eqref{eqt:fixed-point_formulation} corresponds to a solution of our main equations \eqref{eqt:main} via the change of variables \eqref{eqt:change_variable_1}, \eqref{eqt:change_variable_2} (with $c_\om$ given by the second equation in \eqref{eqt:non-degeneracy_condition}).

\section{Existence of solution by a fixed-point method}\label{sec:fixed-point_method}
Our goal of this section is to show that the nonlinear integral equations \eqref{eqt:fixed-point_formulation} admit a non-trivial solution. As we can see, the system \eqref{eqt:fixed-point_formulation} naturally defines a fixed-point problem of a nonlinear nonlocal map. To prove the existence of a suitable fixed-point solution, we need to construct some appropriate function set in a Banach function space on which we can establish continuity and compactness of this nonlinear map, and then we use the Schauder fixed-point theorem (Fact \ref{fact:Schauder}).

\subsection{A fixed-point problem} Consider a Banach space of continuous even functions,
\[\mathbb{V} = \{f\in C(\R):\ f(x) = f(-x),\ \|\rho f\|_{L^\infty}<+\infty\},\]
endowed with a weighted $L^\infty$-norm $\|\rho f\|_{L^\infty}$, referred to as the $L^\infty_\rho$-norm, where $\rho(x) := (1+|x|)^{1+\delta_\rho}$ for some $\delta_\rho>0$ to be defined below. In this topology, we introduce a nonempty, closed, and convex subset of $\mathbb{V}$, in which we will prove the existence of a fixed point:
\begin{equation*}
\begin{split}
\mathbb{D} = \Big\{f\in \mathbb{V}:\quad &f(x)\geq 0,\ f(0) = 1,\ \text{$m_1(x)\leq f(x)\leq 1 $ for all $x$},\\
&f'_{-}(1)\leq -\eta,\ \text{$f(x)$ is non-increasing on $[0,+\infty)$ },\ \text{$f(\sqrt{s})$ is convex in $s$}, \\
&\text{$f(x)\leq \min\big\{(1+3/\delta_1)(|x|/L_1)^{-1-\delta_1}\,,\, 5|x|^{-\delta_0}\big\}$ } \Big\}.
\end{split}
\end{equation*}
Here and below, $f'_-$ and $f'_+$ denote the left and the right derivatives of $f$, respectively. For reasons that will become clear later, we set $\eta := 1/(3^{11}\cdot2^{14}\sqrt{2})$, $\delta_0 := 54\eta/(8+27\eta)$, and
\begin{equation}\label{eqt:m1_definition}
m_1(x) := \frac{1}{g_1(x)}\exp\left(2\int_0^x\frac{1-g_1(y)}{yg_1(y)}\idiff y\right),
\end{equation}
with
\[g_1(x) := \min\left\{ 1+\frac{x^2}{2}\,,\, 1+\frac{3L_0}{4\eta}|x| \right\}\quad \text{and}\quad L_0 := \int_0^{+\infty}\left|t\ln\left|\frac{1+t}{1-t}\right|-2\right|\idiff t<+\infty.\]
Moreover, $L_1>0$ is an absolute constant that will be determined implicitly by $\delta_0$ and $m_1$ in Lemma \ref{lem:L_1}, and $\delta_1>0$ is also an absolute constant that will be determined explicitly by the function $m_1$ in Corollary \ref{cor:R_decay_strong}. Once $\delta_1$ is determined, we choose $\delta_\rho = \delta_1/2$ so that $\mathbb{D}$ is closed and bounded in the $L^\infty_\rho$-norm with $\rho(x) = (1+|x|)^{1+\delta_\rho}$. It is also apparent that $\mathbb{D}$ is a convex set. Later, we will argue that $\mathbb{D}$ is nonempty. In fact, the function $m_1$ defined above belongs to $\mathbb{D}$ by our choice of the constants. 

One can check that $f\in \mathbb{D}$ implies
\begin{equation}\label{eqt:f_range}
(1+x^2/2)^{-2} \leq f(x)\leq \max\big\{1-\eta x^2/2 \,,\, 1-\eta/2\big\}<1,\quad \text{for $x> 0$}.
\end{equation}
The lower bound above owes to the fact that $m_1(x)\geq (1+x^2/2)^{-2}$. Indeed, we have
\begin{equation}\label{eqt:m_1_lower_bound}
m_1(x) = \frac{1}{g_1(x)}\exp\left(2\int_0^x\frac{1-g_1(y)}{yg_1(y)}\idiff y\right) \geq \frac{1}{1+x^2/2}\exp\left(-\int_0^x\frac{2y}{2+y^2}\idiff y\right)  = \frac{4}{(2+x^2)^2}.
\end{equation}
The upper bound of $f$ in \eqref{eqt:f_range} follows from the assumptions that $f(\sqrt{s})$ is convex in $s$, $f'_{-}(1)\leq -\eta$, and $f(x)$ is non-increasing on $[0,+\infty)$.

We remark that, though a function $f\in\mathbb{D}$ is not required to be differentiable, the one-sided derivatives $f'_-(x)$ and $f'_+(x)$ are both well defined at every point $x$ by the convexity of $f(\sqrt{s})$ in $s$. In what follows, we will abuse notation and simply use $f'(x)$ for $f'_-(x)$ and $f'_+(x)$ in both weak sense and strong sense. For example, when we write $f'(x)\leq C$, we mean $f'_-(x)\leq C$ and $f'_+(x)\leq C$ at the same time. In this context, the non-increasing property of $f$ on $[0,+\infty)$ can be represented as $f'(x)\leq 0$ for $x\geq 0$.

Now, we formally construct our nonlinear map in a few steps, guided by the fixed-point problem \eqref{eqt:fixed-point_formulation}. We first define a linear map
\[\mtx{T}(f)(x) := \frac{1}{\pi}\int_0^{+\infty}f(y)\left(\frac{y}{x}\ln\left|\frac{x+y}{x-y}\right|-2\right)\idiff y.\]
One can verify that $\mtx{T}(f)(0)=0$ for $f\in \mathbb{D}\subset L^1(\R)$. Comparing this definition with \eqref{eqt:u_formula}, we find
\begin{equation}\label{eqt:T_to_laplacian}
\mtx{T}(f)(x) = \frac{1}{x}(-\Delta)^{-1/2}(xf)(x) + \mtx{H}(xf)(0) = \frac{1}{x}(-\Delta)^{-1/2}(xf)(x) - b(f),
\end{equation}
where 
\begin{equation}\label{eqt:bf_definition}
b(f) := \frac{2}{\pi}\int_0^{+\infty}f(y)\idiff y.
\end{equation}
Moreover, it is not hard to check that, for $f\in \mathbb{D}$,
\begin{equation}\label{eqt:T_limit}
\lim_{x\rightarrow +\infty}\mtx{T}(f)(x) = \frac{2}{\pi}\int_0^{+\infty}(f(+\infty)-f(y))\idiff y=-\frac{2}{\pi}\int_0^{+\infty}f(y)\idiff y = -b(f).
\end{equation}
Corresponding to the second line of \eqref{eqt:fixed-point_formulation}, we define
\[\mtx{G}(f) := 1 - \frac{\mtx{T}(f)}{c(f)},\]
where we set
\begin{equation}\label{eqt:cf_definition}
c(f) := -\frac{4}{3\pi}\int_0^{+\infty}\frac{f'(y)}{y}\idiff y = \frac{4}{3\pi}\int_0^{+\infty}\frac{f(0)-f(y)}{y^2}\idiff y.
\end{equation}
As mentioned above, this is using the second degree of freedom in the scaling property \eqref{eqt:scaling} to determine the value of $c$ in \eqref{eqt:change_variable_2} in terms of $f$. It will be clear in the proof of Corollary \ref{cor:G_property} below that we choose to define $c(f)$ in this way so that $\lim_{x\rightarrow0}\mtx{G}(f)'(x)/x=1$. Note that $c(f)$ must be strictly positive and finite for any $f\in \mathbb{D}$. Next, in view of the third line of \eqref{eqt:fixed-point_formulation}, we define
\[\mtx{M}(f)(x) = \frac{1}{\mtx{G}(f)(x)}\exp\left(2\int_0^x\frac{1-\mtx{G}(f)(y)}{y\mtx{G}(f)(y)}\idiff y\right).\]
Note that $\mtx{M}(f)(0) = 1/\mtx{G}(f)(0) = 1$ for $f\in \mathbb{D}$. Finally, letting
\begin{equation}\label{eqt:df_definition}
d(f) := \frac{c(f) + b(f)}{2c(f)} = \frac{1}{2} + \frac{b(f)}{2c(f)},
\end{equation}
we define
\[
\mtx{R}(f)(x) = \frac{\displaystyle \int_0^x d(f) y^{d(f)-1}\frac{\mtx{M}(f)(y)}{\mtx{G}(f)(y)}\exp\left(\big(d(f)-1\big)\int_0^y\frac{1-\mtx{G}(f)(z)}{z\mtx{G}(f)(z)}\idiff x\right)\idiff y}{\displaystyle x^{d(f)}\exp\left(\big(d(f)-1\big)\int_0^x\frac{1-\mtx{G}(f)(y)}{y\mtx{G}(f)(y)}\idiff y\right)}.
\]
The rest of this paper aims to study the fixed-point problem 
\[\mtx{R}(f) = f,\quad f\in \mathbb{D}.\]
The following proposition explains how a fixed-point of $\mtx{R}$ relates to a solution of \eqref{eqt:main}.

\begin{proposition}\label{prop:fixed_point_solution}
If $f\in \mathbb{D}$ is a fixed point of $\mtx{R}$, i.e. $\mtx{R}(f) = f$, then $f$ is a solution to the integral equations \eqref{eqt:fixed-point_formulation} with $g = \mtx{G}(f)$, $m = \mtx{M}(f)$, and
\[b = b(f) = \frac{2}{\pi}\int_0^{+\infty}f(y)\idiff y,\quad c=c(f) = \frac{4}{3\pi}\int_0^{+\infty}\frac{1-f(y)}{y^2}\idiff y.\]
As a consequence, $(\om,v,c_l,c_\om)$ is a solution to equations \eqref{eqt:main} with $\om = xf$, $v = c_lx\mtx{M}(f)/2$, and 
\begin{equation}\label{eqt:f_to_cl_cw}
c_l = b(f) + c(f),\quad c_\om = \frac{c(f)-b(f)}{2}.
\end{equation}
\end{proposition}

\begin{proof}
Denote $r = \mtx{R}(f)$, $m = \mtx{M}(f)$, $g = \mtx{G}(f)$, $b=b(f)$, $c=c(f)$, and $d=d(f)$. The first statement follows directly from the construction of these maps. Moreover, it is straightforward to compute that 
\begin{equation}\label{eqt:map_derivatives}
\begin{split}
&xgr' = (1-g)r - dr + dm,\\
&xgm' = 2(1-g)m - xg'm,\\
&g = 1 - \frac{(-\Delta)^{-1/2}(xf)/x-b}{c}.
\end{split}
\end{equation}
Hence, if $f = r$, then $(f,g,m,b,c,d)$ is a solution of the equations \eqref{eqt:f-g-m}. The second statement then follows from the change of variables \eqref{eqt:change_variable_1}, \eqref{eqt:change_variable_2} and the relations in \eqref{eqt:non-degeneracy_condition}.
\end{proof}

We now explain the design of the set $\mathbb{D}$. The guiding idea is to make $\mathbb{D}$ as specific as possible while keeping it compact in the $L^\infty_\rho$-norm and closed under the map $\mtx{R}$, so that we can apply the Schauder fixed-point theorem. The compactness of $\mathbb{D}$ is relatively easy, while its closedness under $\mtx{R}$ takes most of the work. Firstly, the monotonicity of $f(x)$ in $x$ and the convexity of $f(\sqrt{s})$ in $s$ for $f\in \mathbb{D}$, which provide powerful controls on the functions in $\mathbb{D}$, are magically preserved by the map $\mtx{R}$. This is essentially due to the nice properties of the kernel of the linear map $\mtx{T}$.  Secondly, the monotonicity, the convexity, and some calculations of $\mtx{T}$ will lead to the control $m_1(x)\leq \mtx{R}(f) \leq 1$. Hence, we can impose this condition on $\mathbb{D}$ as well. Thirdly, we need $\mathbb{D}\subset L^1(\R)$ in order for $b(f)$ to be well defined on $\mathbb{D}$, which explains the condition $f(x)\lesssim |x|^{-1-\delta_1}$. However, to make sure this tail bound is not compromised by $\mtx{R}$, we need the weaker estimate $f(x)\lesssim |x|^{-\delta_0}$, and to further pass this weaker estimate to $\mtx{R}(f)$ we need to use the condition $f'_{-}(1)\leq -\eta$. Also, $f'_{-}(1)\leq -\eta$ leads to the upper bound in \eqref{eqt:f_range}, which further implies $c(f)>0$, so that $\mtx{G}(f)$ and $d(f)$ are both well-defined. Finally, the estimate $\mtx{R}(f)'(1)\leq -\eta$ can be derived from $f'_{-}(1)\leq -\eta$ and the convexity of $f(x)$ in $x^2$, therefore closing the loop. 

The remainder of this section is devoted to justifying the above arguments and to proving the existence of a fixed point of $\mtx{R}$ in $\mathbb{D}$.

\subsection{Estimates of $b(f)$ and $c(f)$} Recall the functionals $b(f)$, $c(f)$, $d(f)$ defined in \eqref{eqt:bf_definition}, \eqref{eqt:cf_definition}, \eqref{eqt:df_definition}, respectively, which will be constantly used throughout this paper. We start with some estimates on $b(f)$ and $c(f)$ that will be used frequently in what follows.

\begin{lemma}\label{lem:df_bound}
For any $f\in \mathbb{D}$, 
\[\frac{\sqrt{2}}{2}< b(m_1)\leq b(f)\leq \frac{1}{\pi}\left(L_1 + \frac{4}{\delta_1^2}\right),\]
and
\[\frac{4\eta}{3\pi}\leq c(f)\leq c(m_1)<\frac{\sqrt{2}}{2}.\]
As a consequence, 
\[1<d(m_1)\leq d(f) = \frac{1}{2} + \frac{b(f)}{2c(f)}\leq \frac{1}{2} + \frac{3}{2\eta}\left(\frac{L_1}{4} + \frac{1}{\delta_1^2}\right)<+\infty.\]
\end{lemma}

\begin{proof}
Since $f(x)\leq \min\{1\,,\,(1+3/\delta_1)(x/L_1)^{-1-\delta_1}\}$, we can upper bound $b(f)$ as
\[b(f)\leq \frac{1}{\pi}\int_0^{L_1}1\idiff x + \left(1+\frac{3}{\delta_1}\right)\frac{1}{\pi}\int_{L_1}^{+\infty}\left(\frac{x}{L_1}\right)^{-1-\delta_1}\idiff x = \frac{1}{\pi}\left(L_1 + \frac{1}{\delta_1} + \frac{3}{\delta_1^2}\right) \leq \frac{1}{\pi}\left(L_1 + \frac{4}{\delta_1^2}\right).\]
In view of \eqref{eqt:f_range}, we can lower bound $c(f)$ as
\[
c(f) = \frac{4}{3\pi}\int_0^{+\infty}\frac{1-f(y)}{y^2}\idiff y  \geq \frac{4}{3\pi}\int_0^1\frac{\eta}{2}\idiff y + \frac{4}{3\pi}\int_1^{+\infty}\frac{\eta}{2y^2}\idiff y = \frac{4\eta}{3\pi}.
\]

By the definition of $b(f)$ and $c(f)$, it is easy to see that $b(f)\geq b(m_1)$ and $c(f)\leq c(m_1)$ for $f\in \mathbb{D}$. Since $g_1(x)\leq 1+x^2/2$, we have
\[m_1(x) = \frac{1}{g_1(x)}\exp\left(2\int_0^x\frac{1-g_1(y)}{yg_1(y)}\idiff y\right) \geq \frac{1}{1+x^2/2}\exp\left(-\int_0^x\frac{2y}{2+y^2}\idiff y\right)  = \frac{4}{(2+x^2)^2}.\]
In particular, by the definition of $g_1$, the inequality above is strict when $x$ is sufficiently large.
It then follows that
\[
b(m_1)> \frac{2}{\pi}\int_0^{+\infty}\frac{4}{(2+x^2)^2}\idiff y = \frac{\sqrt{2}}{2},
\]
and
\[c(m_1)< \frac{4}{3\pi}\int_0^{+\infty} \frac{4+y^2}{(2+y^2)^2}\idiff y  =\frac{\sqrt{2}}{2}.\]
This proves the lemma.
\end{proof}

The next lemma provides an $x$-dependent bound for $c(f)$ that will later be used in the proof of the estimate $\mtx{R}(f)'(1)\leq -\eta$.

\begin{lemma}\label{lem:cf_bound}
For any $f\in \mathbb{D}$ and any $x>0$, 
\[c(f) \leq \frac{8(x+1)}{3\pi x}\big(1-f(x)\big)^{1/2}.\]
\end{lemma}

\begin{proof}
Fix an $x>0$. For $0\leq y\leq x$, $f(y)\geq \max\{(1+y^2/2)^{-2}, f(x)\}\geq \max\{1-y^2, f(x)\}$, so $1-f(y)\leq \min\{y^2, 1-f(x)\}$. For $y>x$, the convexity of $f(\sqrt{s})$ in $s$ implies that $(1-f(y))/y^2\leq (1-f(x))/x^2$, and so $1-f(y)\leq \min\{y^2(1-f(x))/x^2, 1\}$. Combining these estimates yields 
\[1-f(y)\leq \min\{y^2, 1-f(x)\} + \min\{y^2(1-f(x))/x^2, 1\}, \quad y\geq0.\]
We thus obtain that 
\begin{align*}
c(f) &\leq \frac{4}{3\pi}\int_0^{+\infty}\frac{\min\{y^2, 1-f(x)\}}{y^2}\idiff y + \frac{4}{3\pi}\int_0^{+\infty}\frac{\min\{y^2(1-f(x))/x^2, 1\}}{y^2}\idiff y\\
&\leq \frac{8(x+1)}{3\pi x}\big(1-f(x)\big)^{1/2}, 
\end{align*}
which is the desired bound.
\end{proof}

We will need the continuity of $b(f)$ and $c(f)$ for proving the continuity of $\mtx{R}$ in the $L^\infty_\rho$-norm.

\begin{lemma}\label{lem:df_continuous}
$b(f), c(f), d(f):\mathbb{D}\to \R$ are all continuous in the $L^\infty_\rho$-norm. In particular, 
\[|b(f_1) - b(f_2)|\lesssim \|f_1-f_2\|_{L^\infty_\rho},\]
\[|c(f_1) - c(f_2)|\lesssim \|f_1-f_2\|_{L^\infty_\rho}^{1/2},\]
and 
\[|d(f_1) - d(f_2)|\lesssim \|f_1-f_2\|_{L^\infty_\rho}^{1/2},\]
for any $f_1,f_2\in \mathbb{D}$.
\end{lemma}

\begin{proof}
Recall that $\rho(x) = (1+|x|)^{1+\delta_\rho}$. Denote $\delta = \|f_1-f_2\|_{L^\infty_\rho} = \|\rho(f_1-f_2)\|_{L^\infty}$. Since
\[ \frac{1}{(1+x^2/2)^2}\leq f_i(x)\leq 1,\quad i=1,2, \]
we have
\[|f_1(x)-f_2(x)|\leq \min\left\{x^2\,,\,\delta(1+|x|)^{-1-\delta_\rho}\right\}.\] 
Hence, 
\[
|b(f_1) - b(f_2)| \lesssim \frac{1}{\pi}\int_0^{+\infty}|f_1(x)-f_2(x)|\idiff x \leq \delta\int_0^{+\infty}(1+|x|)^{-1-\delta_\rho}\idiff x = \frac{\delta}{\delta_\rho},
\]
and 
\[|c(f_1)-c(f_2)|\lesssim \int_0^{+\infty}\frac{|f_1(x)-f_2(x)|}{x^2}\idiff x \leq \int_0^{\sqrt{\delta}}1\idiff x + \delta\int_{\sqrt{\delta}}^{+\infty}\frac{(1+|x|)^{-1-\delta_\rho}}{x^2}\idiff x\lesssim \sqrt{\delta}.\]
The continuity of $d(f)$ then follows from the continuity of $b(f),c(f)$ and Lemma \ref{lem:df_bound}.
\end{proof}

\subsection{Monotonicity and convexity}
Next, we derive the key property of $\mtx{R}$, that it preserves the monotonicity in $x$ and convexity in $x^2$ on $[0,+\infty)$ for functions in $\mathbb{D}$. This nice property of $\mtx{R}$ essentially arises from a kernel analysis of the linear map $\mtx{T}$ as shown in the following lemma, which is a restatement of its counterpart in \cite{huang2024self}. We still provide the proof here for the sake of completeness.

\begin{lemma}\label{lem:T_property}
Given $f\in \mathbb{D}$, $\mtx{T}(f)'(x)\leq 0$ on $[0,+\infty)$, and $\mtx{T}(f)(\sqrt{s})$ is convex in $s$. 
\end{lemma}

\begin{proof}
We first show that $\mtx{T}(f)'(x)\leq 0$ on $(0,+\infty)$. We can use integration by parts to compute that, for $x>0$,  
\begin{equation}\label{eqt:T_integrate_by_part}
\begin{split}
\mtx{T}(f)(x) &= \frac{1}{\pi}\int_0^{+\infty}f(y)\left(\frac{y}{x}\ln\left|\frac{x+y}{x-y}\right|-2\right)\idiff y\\
&= \frac{1}{\pi}\int_0^{+\infty}f(y)\cdot \partial_y\left(\frac{y^2-x^2}{2x}\ln\left|\frac{x+y}{x-y}\right|-y\right)\idiff y\\
&= \frac{1}{\pi}\int_0^{+\infty}f'(y)\cdot\left(\frac{x^2-y^2}{2x}\ln\left|\frac{x+y}{x-y}\right|+y\right)\idiff y\\
&= \frac{1}{\pi}\int_0^{+\infty}f'(y)\cdot yF_1(x/y)\idiff y,
\end{split}
\end{equation}
where the function $F_1$ is defined in \eqref{eqt:F1_definition} in Appendix \ref{sec:F_1}, and the integration by parts can be justified by the properties of $F_1$ proved in Lemma \ref{lem:F1_property}. Therefore, we have
\begin{equation}\label{eqt:T_derivative}
\mtx{T}(f)'(x) = \frac{1}{\pi}\int_0^{+\infty}f'(y)\cdot y\partial_x F_1(x/y)\idiff y = \frac{1}{\pi}\int_0^{+\infty}f'(y)\cdot F_1'(x/y)\idiff y \leq 0,
\end{equation}
where the inequality follows from property (3) in Lemma \ref{lem:F1_property}.

Next, we show that $\mtx{T}(f)(\sqrt{s})$ is convex in $s$. By approximation theory, we may assume that $f(\sqrt{s})$ is twice differentiable in $s$, so that the convexity of $f(\sqrt{s})$ in $s$ is equivalent to $(f'(x)/x)'\geq 0$ for $x>0$. Continuing the calculations above, we have
\begin{align*}
\frac{\mtx{T}(f)'(x)}{x} &= \frac{1}{\pi}\int_0^{+\infty}\frac{f'(y)}{y}\cdot \frac{y}{x}\left(\frac{y^2+x^2}{2x^2}\ln\left|\frac{x+y}{x-y}\right| - \frac{y}{x}\right)\idiff y\\
&= \frac{1}{\pi}\int_0^{+\infty}\frac{f'(y)}{y}\cdot \partial_y\left(\frac{y^4+2x^2y^2-3x^4}{8x^3}\ln\left|\frac{x+y}{x-y}\right| - \frac{y^3}{4x^2} - \frac{7y}{12}\right)\idiff y + \frac{4}{3\pi}\int_0^{+\infty}\frac{f'(y)}{y}\idiff y\\
&= \frac{1}{\pi}\int_0^{+\infty}\left(\frac{f'(y)}{y}\right)'\cdot\left(\frac{3x^4-2x^2y^2-y^4}{8x^3}\ln\left|\frac{x+y}{x-y}\right| + \frac{y^3}{4x^2} + \frac{7y}{12}\right)\idiff y + \frac{4}{3\pi}\int_0^{+\infty}\frac{f'(y)}{y}\idiff y\\
&= \frac{1}{\pi}\int_0^{+\infty}\left(\frac{f'(y)}{y}\right)'\cdot y F_2(x/y)\idiff y + \frac{4}{3\pi}\int_0^{+\infty}\frac{f'(y)}{y}\idiff y.
\end{align*}
where the function $F_2$ is defined in \eqref{eqt:F2_definition} in Appendix \ref{sec:F_2}, and the integration by parts can be justified by the properties of $F_2$ proved in Lemma \ref{lem:F2_property}. Therefore, 
\[
\left(\frac{\mtx{T}(f)'(x)}{x}\right)' = \frac{1}{\pi}\int_0^{+\infty}\left(\frac{f'(y)}{y}\right)'\cdot y \partial_x F_2(x/y)\idiff y = \frac{1}{\pi}\int_0^{+\infty}\left(\frac{f'(y)}{y}\right)'\cdot F_2'(x/y)\idiff y\geq 0,
\]
where the last inequality follows from property (3) in Lemma \ref{lem:F2_property}. This implies the convexity of $\mtx{T}(f)(\sqrt{s})$ in $s$.
\end{proof}

A similar property holds for the map $\mtx{G}$ as it is just $1$ plus a negative multiple of $\mtx{T}$. This then leads to a crude but useful universal estimate of $\mtx{G}(f)$ for $f\in \mathbb{D}$.

\begin{corollary}\label{cor:G_property}
For any $f\in \mathbb{D}$, $\mtx{G}(f)'(x)\geq 0$ on $[0,+\infty)$, and $\mtx{G}(f)(\sqrt{s})$ is concave in $s$. As a consequence, $1\leq \mtx{G}(f)(x)\leq \min\{1+x^2/2\,,\,2d(f)\}$ for all $x$.
\end{corollary}

\begin{proof}
Write $g=\mtx{G}(f)$. The monotonicity of $g(x)$ and the concavity of $g(\sqrt{s})$ follow directly from Lemma \ref{lem:T_property}. More precisely, we have $g'(x)\geq 0$ and $(g(x)'/x)'\leq 0$ for all $x\geq 0$. By monotonicity we immediately have that 
\[1=g(0) \leq g(x)\leq g(+\infty) = 1 + \frac{b(f)}{c(f)} = 2d(f).\]
We have used \eqref{eqt:T_limit} for the limit value $g(+\infty)$.

Next, we show that $g(x)\leq 1+x^2/2$ for all $x$. In fact, we can compute that
\[
\lim_{x\rightarrow 0}\frac{g'(x)}{x} = -\frac{1}{c(f)}\lim_{x\rightarrow 0}\frac{\mtx{T}(f)'(x)}{x} = -\frac{1}{c(f)}\cdot \frac{4}{3\pi}\int_0^{+\infty}\frac{f'(y)}{y}\idiff y = 1.
\]
Then, by the concavity of $g(\sqrt{s})$ in $s$, we have for $x>0$,
\[\frac{g(x)-1}{x^2} = \frac{g(x)-g(0)}{x^2} \leq \lim_{x\rightarrow 0}\frac{g'(x)}{2x} = \frac{1}{2},\]
which is the desired bound.
\end{proof}

The above properties of $\mtx{G}$ imply similar properties of $\mtx{M}$ through some straightforward derivative calculations.

\begin{corollary}\label{cor:M_property}
For any $f\in \mathbb{D}$, $\mtx{M}(f)'(x)\leq 0$ on $[0,+\infty)$, and $\mtx{M}(f)(\sqrt{s})$ is convex in $s$.
\end{corollary}

\begin{proof}
Write $m = \mtx{M}(f)$ and $g= \mtx{G}(f)$. From \eqref{eqt:map_derivatives} we find that
\begin{align*}
m'(x) = -\left(\frac{g'(x)}{g(x)} + 2\frac{g(x)-1}{xg(x)}\right)m(x)\leq 0.
\end{align*}
We have used $g'(x)\geq 0$ and $g(x)\geq g(0) = 1$. 

To prove $m(\sqrt{s})$ is convex in $s$, we only need to show that $(m'(x)/x)'\geq 0$ for $x>0$. Note that the concavity of $g(\sqrt{s})$ in $s$ implies that, for $x>0$,
\[\frac{g(x)-1}{x^2}\geq \frac{g'(x)}{2x},\]
which further implies 
\[\left(\frac{g(x)-1}{x^2}\right)' = \frac{g'(x)}{x^2} - \frac{2(g(x)-1)}{x^3}\leq 0.\]
Therefore, for $x>0$, we can compute
\begin{align*}
\left(\frac{m'(x)}{x}\right)' &= -\left(\left(\frac{g'(x)}{x}\right)' + 2\left(\frac{g(x)-1}{x^2}\right)'\right)\cdot \frac{m(x)}{g(x)} - 2\left(\frac{g'(x)}{g(x)} + \frac{g(x)-1}{xg(x)}\right)\cdot\frac{m'(x)}{x}\\
&\geq -\left(\left(\frac{g'(x)}{x}\right)' + 2\left(\frac{g(x)-1}{x^2}\right)'\right)\cdot \frac{m(x)}{g(x)}\\
&\geq 0,
\end{align*}
as desired.
\end{proof}

Finally, we show that the map $\mtx{R}$ also enjoys the same monotone and convex properties.

\begin{corollary}\label{cor:R_property}
For any $f\in \mathbb{D}$, $\mtx{R}(f)'(x)\leq0$ on $[0,+\infty)$, and $\mtx{R}(f)(\sqrt{s})$ is convex in $s$. Moreover, $\mtx{R}(f)(x)\geq \mtx{M}(f)(x)$ for all $x$.
\end{corollary}

\begin{proof} Let $r = \mtx{R}(f), m = \mtx{M}(f), g= \mtx{G}(f), d=d(f)$. For $x\geq 0$, define
\begin{equation}\label{eqt:A_B}
\begin{split}
A(x) &:= \int_0^x dy^{d-1}\cdot\frac{m(y)}{g(y)}\exp\left(\big(d-1\big)\int_0^y\frac{1-g(z)}{zg(z)}\idiff z\right)\idiff y,\\
B(x) &:= x^d\exp\left(\big(d-1\big)\int_0^x\frac{1-g(y)}{yg(y)}\idiff y\right),
\end{split}
\end{equation}
so that $r(x) = A(x)/B(x)$. Note that $A(x)\geq 0, B(x)\geq 0$, $A(0) = B(0) = 0$, and
\[r(0) = \lim_{x\rightarrow 0}\frac{A(x)}{B(x)} = \frac{m(0)}{g(0)} = 1.\] 
We can compute that 
\[A'(x) = dx^{d-1}\frac{m(x)}{g(x)}\exp\left(\big(d-1\big)\int_0^x\frac{1-g(y)}{yg(y)}\idiff y\right) = d\frac{m(x)}{g(x)}\cdot\frac{B(x)}{x}\geq 0,\]
\[B'(x) = \left(d + (d-1)\frac{1-g(x)}{g(x)}\right)\cdot \frac{B(x)}{x} = \left(1 -\frac{1}{g(x)} + \frac{d}{g(x)}\right)\cdot \frac{B(x)}{x}\geq 0,\]
and thus
\[A'(x) = \frac{dm(x)}{g(x) + d -1}\cdot B'(x).\]
Note that, by Corollaries \ref{cor:G_property} and \ref{cor:M_property}, the function $dm(x)/(g(x) + d - 1)$ is non-increasing in $x$. Hence, 
\[A(x) = \int_0^x\frac{dm(y)}{g(y) + d - 1}\cdot B'(y)\idiff y  \geq \frac{dm(x)}{g(x) + d - 1} \cdot \int_0^xB'(y)\idiff y = \frac{dm(x)}{g(x) + d - 1}B(x),\]
that is,
\begin{equation*}
r(x) =\frac{A(x)}{B(x)} \geq\frac{dm(x)}{g(x) + d - 1} = \frac{A'(x)}{B'(x)}.
\end{equation*}
We then find that, for $x>0$, 
\[r'(x) = \left(\frac{A(x)}{B(x)}\right)' = \frac{B'(x)}{B(x)}\left(\frac{A'(x)}{B'(x)}-\frac{A(x)}{B(x)}\right)\leq 0.\]
For the record, we also compute that
\begin{equation}\label{eqt:r_derivative}
r'(x) = \frac{dm(x)-(d+g(x)-1)r(x)}{xg(x)}.
\end{equation}
From this we find that $r'(0) = (dm'(0) - g'(0))/(1+d)=0$. 

Next, we show that $\mtx{R}(f)(\sqrt{s})$ is convex in $s$. To this end, we first show that $(r(x) - m(x))/x^2$ is non-increasing in $x$ in an analogous way. Define
\begin{equation}\label{eqt:tilde_A_B}
\tilde{A}(x) := A(x) - B(x)m(x),\quad \tilde{B}(x) := x^2B(x).
\end{equation}
Consider the function
\[
\frac{r(x)-m(x)}{x^2} = \frac{B(x)r(x)-B(x)m(x)}{x^2B(x)} = \frac{\tilde{A}(x)}{\tilde{B}(x)}.
\]
We find that for $x\geq 0$,
\[\tilde{A}'(x) = A'(x) - B'(x)m(x) -B(x)m'(x) = \left(g'(x) + \frac{g(x)-1}{x}\right)\cdot\frac{B(x)m(x)}{g(x)}\geq 0,\]
and
\[\tilde{B}'(x) = 2xB(x) + x^2B'(x)=\left(3+\frac{d-1}{g(x)}\right)\cdot xB(x)\geq 0.\]
These calculations imply $\tilde{A}(x),\tilde{B}(x)\geq 0$ since $\tilde{A}(0)=\tilde{B}(0)=0$. It follows that
\[r(x) - m(x) = x^2\cdot \frac{\tilde{A}(x)}{\tilde{B}(x)} \geq 0.\] 
This concludes that $\mtx{R}(f)(x)\geq \mtx{M}(f)(x)$ for all $x$. From the above we also find that 
\[
\frac{\tilde{A}'(x)}{\tilde{B}'(x)} = \left(\frac{g'(x)}{x} + \frac{g(x)-1}{x^2}\right)\cdot \frac{m(x)}{3g(x) + d -1}.
\]
By the concavity of $g(\sqrt{s})$ in $s$ (Corollary \ref{cor:G_property}), $g'(x)/x$ and $(g(x)-1)/x^2$ are both non-increasing in $x$. Moreover, the function $m(x)/(3g(x)+d-1)$ is also non-increasing in $x$. Hence, $\tilde{A}'(x)/\tilde{B}'(x)$ is non-increasing in $x$. This again implies that $\tilde{A}(x)/\tilde{B}(x)\geq \tilde{A}'(x)/\tilde{B}'(x)$, and thus, for $x>0$,
\[\left(\frac{r(x)-m(x)}{x^2}\right)' =  \left(\frac{\tilde{A}(x)}{\tilde{B}(x)}\right)' = \frac{\tilde{B}'(x)}{\tilde{B}(x)}\left(\frac{\tilde{A}'(x)}{\tilde{B}'(x)}-\frac{\tilde{A}(x)}{\tilde{B}(x)}\right)\leq 0.\]
Now, we can compute that 
\begin{align*}
\left(\frac{r'(x)}{x}\right)' &= \left(\frac{d(m(x)-r(x)) - (g(x)-1)r(x)}{x^2g(x)}\right)'\\
&= -\frac{g'(x)}{g(x)^2}\cdot \frac{d(m(x)-r(x)) - (g(x)-1)r(x)}{x^2} - \frac{d}{g(x)}\cdot\left(\frac{r(x)-m(x)}{x^2}\right)'\\
&\qquad  - \left(\frac{g(x)-1}{x^2}\right)'\cdot\frac{r(x)}{g(x)} - \frac{g(x)-1}{x^2g(x)}\cdot r'(x)\\
&= -\left(g'(x) + \frac{g(x)-1}{x}\right)\cdot \frac{r'(x)}{xg(x)} - \frac{d}{g(x)}\cdot\left(\frac{r(x)-m(x)}{x^2}\right)' - \left(\frac{g(x)-1}{x^2}\right)'\cdot\frac{r(x)}{g(x)}\\
&\geq 0.
\end{align*}
We have also used Corollary \ref{cor:G_property} for the last inequality. This completes the proof.
\end{proof}

We proceed to show that $\mtx{R}$ preserves the uniform lower bound $f(x)\geq m_1(x)$ appearing in the definition of the set $\mathbb{D}$. To this end, we first prove a finer uniform estimate for $\mtx{G}(f)$ for all $f\in \mathbb{D}$ that strictly improves the bound $\mtx{G}(f)(x)\leq 1+x^2/2$ (in Corollary \ref{cor:G_property}) for sufficiently large $x$.

\begin{lemma}\label{lem:g_g1}
For any $f\in \mathbb{D}$ and for all $x$, 
\[\mtx{G}(f)(x)\leq g_1(x) = \min\left\{ 1+\frac{x^2}{2}\,,\, 1+\frac{3L_0}{4\eta}|x| \right\},\]
where
\[L_0 := \int_0^{+\infty}\left|t\ln\left|\frac{1+t}{1-t}\right|-2\right|\idiff t<+\infty.\]
\end{lemma}

\begin{proof}
Using $f(x)\leq 1$, we can estimate that 
\[|\mtx{T}(f)(x)|\leq \frac{1}{\pi}\int_0^{+\infty}\left|\frac{y}{x}\ln\left|\frac{x+y}{x-y}\right|-2\right|\idiff y = \frac{x}{\pi}\int_0^{+\infty}\left|t\ln\left|\frac{1+t}{1-t}\right|-2\right|\idiff t = \frac{L_0}{\pi}x.\]
By Lemma \ref{lem:df_bound} we have $c(f)\geq 4\eta/3\pi$. It follows that, for $x\geq 0$, 
\[g(x) = 1 - \frac{\mtx{T}(x)}{c(f)}\leq 1 + \frac{3L_0}{4\eta}x.\]
We have shown in Corollary \ref{cor:G_property} that $g(x)\leq 1+x^2/2$. Hence, $g(x)\leq g_1(x)$ for all $x$.
\end{proof}

Comparing the definition of $\mtx{M}(f)$ and that of $m_1$, we immediately have the following.

\begin{corollary}\label{cor:R_M_lower_bound}
For any $f\in \mathbb{D}$, 
\[\mtx{R}(f)(x)\geq \mtx{M}(f)(x)\geq m_1(x) = \frac{1}{g_1(x)}\exp\left(2\int_0^x\frac{1-g_1(y)}{yg_1(y)}\idiff y\right).\]
\end{corollary}

\begin{proof}
Let $g = \mtx{G}(f), m = \mtx{M}(f)$, $r=\mtx{R}(f)$. Since $g(x)\leq g_1(x)$ for all $x$ (Lemma \ref{lem:g_g1}), we have
\begin{align*}
m(x) &= \frac{1}{g(x)}\exp\left(2\int_0^x\frac{1-g(y)}{yg(y)}\idiff y\right) \geq \frac{1}{g_1(x)}\exp\left(2\int_0^x\frac{1-g_1(y)}{yg_1(y)}\idiff y\right)= m_1(x).
\end{align*}
Moreover, we have shown in the proof of Corollary \ref{cor:R_property} that $r(x)\geq m(x)$ for all $x$. This proves the corollary.
\end{proof}

Next, for the particular value of $\eta$ chosen in the definition of $\mathbb{D}$, we will show that $\mtx{R}(f)'(1)\leq -\eta$ for $f\in \mathbb{D}$, which means the property $f'(1)\leq -\eta$ is preserved by $\mtx{R}$. To achieve this, we need the following lemma. 

\begin{lemma}\label{lem:g_1}
For any $f\in \mathbb{D}$, $\mtx{G}(f)'(1)\geq 27\eta/4$. As a consequence, $\mtx{G}(f)(1)\geq 1+ 27\eta/8$.
\end{lemma}

\begin{proof}
Write $g = \mtx{G}(f)$. We first upper bound $\mtx{T}(f)'(x)$ in two ways using some estimates in \cite{huang2024self} (more precisely, in the proof of \cite[Lemma 3.6]{huang2024self}). We restate the detailed calculations below for the reader's convenience. On the one hand, we can use the calculations in the proof of Lemma \ref{lem:T_property} to get that, for $x>0$, 
\begin{align*}
\mtx{T}(f)'(x) &= \frac{1}{\pi}\int_0^{+\infty}f'(y)\cdot F_1'(x/y)\idiff y \leq \frac{1}{\pi}\int_0^{x}f'(y)\cdot F_1'(x/y)\idiff y\\
&\leq \frac{1}{\pi}\int_0^{x}\frac{f'(x)}{x}\cdot yF_1'(x/y)\idiff y= xf'(x)\cdot\frac{1}{\pi}\int_0^1tF_1'(1/t)\idiff t = \frac{xf'(x)}{2\pi}.
\end{align*}
We have used the fact that $\int_0^1tF_1'(1/t)\idiff t = (4t/3-tF_2(1/t))\big|_0^1 = 1/2$ (Lemma \ref{lem:F2_property}). Recall that the special functions $F_1$ and $F_2$ are defined in Appendixes \ref{sec:F_1} and \ref{sec:F_2}, respectively. On the other hand, for any $0<z<x$, we use $F_1'(1/t)\geq 4t^3/3$ for $t\in[0,1]$ to find that 
\begin{align*}
\mtx{T}(f)'(x) &\leq \frac{1}{\pi}\int_0^{x}f'(y)\cdot \frac{4y^3}{3x^3}\idiff y
\leq \frac{1}{\pi}\int_z^{x}f'(y)\cdot \frac{4z^3}{3x^3}\idiff y\\
&= \frac{4z^3}{3\pi x^3}(f(x)-f(z))\leq \frac{4z^3}{3\pi x^3}(f(x)-1 + z^2).
\end{align*}
We then choose $z = \big((1-f(x))/2\big)^{1/2}$ to obtain 
\begin{equation}\label{eqt:T_derivative_bound}
\mtx{T}(f)'(x) \leq -\frac{1}{3\sqrt{2}\pi x^3}\cdot(1-f(x))^{5/2}. 
\end{equation}
Combining the two bounds above, we reach
\[-\mtx{T}(f)'(x) \geq \left(\frac{1}{3\sqrt{2}\pi x^3}\right)^{1/5}(1-f(x))^{1/2}\cdot \left(\frac{x|f'(x)|}{2\pi}\right)^{4/5} = \frac{1}{\pi}\left(\frac{x}{48\sqrt{2}}\right)^{1/5}|f'(x)|^{4/5}(1-f(x))^{1/2}.\]
It follows that, for $x\geq 0$, 
\[g'(x) = \frac{-\mtx{T}'(f)(x)}{c(f)}\geq  \frac{3x}{8(1+x)}\cdot\left(\frac{x}{48\sqrt{2}}\right)^{1/5}|f'(x)|^{4/5}.\]
We have used the upper bound of $c(f)$ in Lemma \ref{lem:cf_bound}. Plugging in $x=1$ and using $f'(1)\leq -\eta$, we get
\[g'(1)\geq \frac{3\eta^{4/5}}{16\cdot(48\sqrt{2})^{1/5}} = \frac{27}{4}\eta.\]
We then use the concavity of $g(\sqrt{s})$ in $s$ to obtain 
\[g(1) - 1 \geq \frac{g'(1)}{2} = \frac{27}{8}\eta.\]
That is, $g(1)\geq 1+ 27\eta/8$. 
\end{proof}

The above lemma will also be needed when we establish uniform decay bounds for $\mtx{R}(f)$ in the next subsection. For now, we use Lemma \ref{lem:g_1} to derive the following.

\begin{corollary}\label{cor:R_1}
For any $f\in \mathbb{D}$, $\mtx{R}(f)'(1)\leq -\eta$.
\end{corollary}

\begin{proof}
Let $g = \mtx{G}(f), m = \mtx{M}(f)$, and $r=\mtx{R}(f)$. By Corollary \ref{cor:G_property} we have $g(x)\leq 1+x^2/2$, and by Corollary \ref{cor:R_M_lower_bound} we have $r(x)\geq m(x)\geq m_1(x)\geq (1+x^2/2)^{-2}$. We then find that, for $x\in[0,1]$,
\begin{align*}
r'(x) &= -\frac{g(x)-1}{x}\cdot \frac{r(x)}{g(x)} - \frac{d}{g(x)}\frac{r(x)-m(x)}{x}\leq -\frac{g(x)-1}{x}\cdot \frac{r(x)}{g(x)}\\
&\leq -\frac{g'(x)}{2}\cdot \frac{r(x)}{g(x)} \leq - \frac{4}{(2+x^2)^3}\cdot g'(x).
\end{align*}
It then follows from Lemma \ref{lem:g_1} that 
\[r'(1) \leq -\frac{4}{27}g'(1) \leq -\eta, \]
as desired.
\end{proof}

\subsection{Decay estimates}
In this subsection, we show that the map $\mtx{R}$ preserves the uniform decay bounds for $f\in \mathbb{D}$, i.e. $\mtx{R}(f)(x)\lesssim |x|^{-\delta_0}$ and $\mtx{R}(f)(x)\lesssim |x|^{-1-\delta_1}$. These estimates will be established in a general framework as presented in the following lemma.

\begin{lemma}\label{lem:g_L}
Given $f\in \mathbb{D}$, if $\mtx{G}(f)(L)\geq 1+a$ for some $L>0$ and some $a\in(0,(5d-1)/(3d+1)]$, then for all $x>0$, 
\[\exp\left(\int_0^x\frac{1-\mtx{G}(f)(y)}{y\mtx{G}(f)(y)}\idiff y\right)\leq \left(\frac{x}{L}\right)^{-a/(1+a)},\]
\[\mtx{M}(f)(x)\leq \left(\frac{x}{L}\right)^{-2a/(1+a)},\]
and 
\[\mtx{R}(f)(x)\leq \left(1+\min\left\{\frac{3(1+a)}{(1-a)_+}\,,\,\frac{12d}{d-1}\right\}\right)\cdot\left(\frac{x}{L}\right)^{-2a/(1+a)}.\]
Here and below, $(t)_+:=\max\{t\,,0\}$.
\end{lemma}

\begin{proof}
Let $r = \mtx{R}(f)$, $m=\mtx{M}(f)$, $g = \mtx{G}(f)$, and $d=d(f)$. Bear in mind that $g(x)$ is non-decreasing on $[0,+\infty)$, and $1=g(0)\leq g(x)\leq g(+\infty)=2d$. Define
\[\psi(x) = \exp\left(\int_0^x\frac{1-g(y)}{yg(y)}\idiff y\right).\]
For $x\geq L$, $g(x)\geq g(L)\geq 1+a$, and thus
\[\psi'(x) = \frac{1-g(x)}{xg(x)}\psi(x) \leq -\frac{a}{1+a}\cdot \frac{\psi(x)}{x}.\]
This implies that, for $x\geq L$,
\[\psi(x)\leq \psi(L)\cdot \left(\frac{x}{L}\right)^{-a/(1+a)}\leq \left(\frac{x}{L}\right)^{-a/(1+a)}.\]
We have used that $\psi(L)\leq \psi(0)=1$. Note that for $x\in(0,L)$ we also have
\[\psi(x)\leq 1 \leq \left(\frac{x}{L}\right)^{-a/(1+a)}.\]
Moreover, this implies 
\[m(x) = \frac{\psi(x)^2}{g(x)} \leq \psi(x)^2 \leq \left(\frac{x}{L}\right)^{-2a/(1+a)}.\]

Let $A(x),B(x)$ be defined as in \eqref{eqt:A_B}, and let $\tilde{A}(x)$ be defined as in \eqref{eqt:tilde_A_B}. We can compute that, 
\[B'(x) = \left(1 + \frac{d-1}{g(x)}\right)\cdot \frac{B(x)}{x}\geq \left(1 + \frac{d-1}{2d}\right)\cdot \frac{B(x)}{x}.\]
This implies, for any $x\geq y>0$, 
\[\frac{B(x)}{B(y)}\geq \left(\frac{x}{y}\right)^{(3d-1)/2d}.\] 
Moreover, we have
\[\tilde{A}'(x) = \left(g'(x) + \frac{g(x)-1}{x}\right)\cdot\frac{B(x)m(x)}{g(x)}\leq \frac{3(g(x)-1)}{xg(x)}\cdot B(x)m(x)\leq \frac{3B(x)m(x)}{x}.\]
We have used the concavity of $g(\sqrt{s})$ in $s$. It follows that
\begin{align*}
\tilde{A}(x) &\leq 3\int_0^x\frac{B(y)m(y)}{y}\idiff y\leq \frac{3B(x)}{x^{(3d-1)/2d}}\int_0^x\frac{m(y)}{y}\cdot y^{(3d-1)/2d}\idiff y\\
&\leq \frac{3B(x)}{x^{(3d-1)/2d}}\int_0^x\left(\frac{y}{L}\right)^{-2a/(1+a)}\cdot y^{(d-1)/2d}\idiff y= \frac{3B(x)}{1+\frac{d-1}{2d}-\frac{2a}{1+a}}\cdot \left(\frac{x}{L}\right)^{-2a/(1+a)}.
\end{align*}
Note that, since $d>1$ (Lemma \ref{lem:df_bound}), for $a \in(0,(5d-1)/(3d+1)]$,
\[1+\frac{d-1}{2d}-\frac{2a}{1+a} \geq \max\left\{\frac{(1-a)_+}{1+a}\,,\,\frac{d-1}{4d}\right\}>0,\]
We then have
\[\tilde{A}(x)\leq \min\left\{\frac{1+a}{(1-a)_+}\,,\,\frac{4d}{d-1}\right\}\cdot 3B(x)\cdot \left(\frac{x}{L}\right)^{-2a/(1+a)}.\]
and thus 
\[r(x) = \frac{A(x)}{B(x)} = m(x) + \frac{\tilde{A}(x)}{B(x)}\leq \left(1+\min\left\{\frac{3(1+a)}{(1-a)_+}\,,\,\frac{12d}{d-1}\right\}\right)\cdot\left(\frac{x}{L}\right)^{-2a/(1+a)}.\]
This completes the proof.
\end{proof}

To apply Lemma \ref{lem:g_L}, we need to establish $\mtx{G}(f)(L)\geq 1+a$ for some $L,a>0$. In fact, we have already obtained this type of condition in Lemma \ref{lem:g_1}, which leads to the first uniform decay bound for $\mtx{R}(f)$.

\begin{corollary}\label{cor:R_decay_weak}
For any $f\in \mathbb{D}$ and for all $x\geq 0$, 
\[\mtx{M}(f)(x)\leq x^{-\delta_0},\]
and
\[\mtx{R}(f)(x)\leq 5x^{-\delta_0},\]
where $\delta_0 = 54\eta/(8+27\eta)$.
\end{corollary}

\begin{proof} By Lemma \ref{lem:g_1} we have $g(1)\geq 1+ 27\eta/8$. We can thus apply Lemma \ref{lem:g_L} with $L=1$ and $a = 27\eta/8$ to obtain 
\[\mtx{M}(f)(x)\leq \left(\frac{x}{L}\right)^{-2a/(1+a)} = x^{-\delta_0},\]
and
\[\mtx{R}(f)(x)\leq \left(1+\frac{3(1+a)}{1-a}\right)\cdot\left(\frac{x}{L}\right)^{-2a/(1+a)}\leq 5x^{-\delta_0}.\]
We have used the fact $a=27\eta/8<1/7$ so that $3(1+a)/(1-a)<4$.
\end{proof}

Next, we use the condition $f(x)\leq 5|x|^{-\delta_0}$ to compute a new pair of $(L,a)$ for the assumption of Lemma \ref{lem:g_L}, so that we can derive a stronger decay bound for $\mtx{R}(f)$.  

\begin{lemma}\label{lem:L_1}
There is some absolute constant $L_1>0$ (that is implicitly determined by $\delta_0$ and $m_1$) such that, for all $f\in \mathbb{D}$, 
\[\mtx{G}(f)(L_1)\geq 1 + d(m_1).\]
\end{lemma}

\begin{proof}
For $t\geq 0$, define 
\begin{equation}\label{eqt:phi}
\phi(t) = 2-t\ln\left|\frac{1+t}{1-t}\right|.
\end{equation}
It is not hard to show that there is some $t_0\in(0,1)$ such that $\phi(t_0)=0$, $0<\phi(t)\leq 2$ for $t\in[0,t_0)$, and $\phi(t)<0$ for $t>t_0$. Also note that
\[\int_0^{+\infty} \phi(t) \idiff t =\int_0^{+\infty} \left(tF_1(1/t)\right)'\idiff t= \lim_{t\rightarrow +\infty}tF_1(1/t) - \lim_{t\rightarrow 0}tF_1(1/t)=0,\]
where $F_1(t)$ is defined in \eqref{eqt:F1_definition}, and the limits above are given in Lemma \ref{lem:F1_property}. 

We then find that, for $x\geq 0$,
\begin{align*}
-\mtx{T}(f)(x) &= \frac{1}{\pi}\int_0^{+\infty}f(y)\phi(y/x)\idiff y = \frac{1}{\pi}\int_0^{+\infty}(f(y)-f(t_0x))\phi(y/x)\idiff y\\
&\geq \frac{1}{\pi}\int_0^{t_0x}(f(y)-f(t_0x))\phi(y/x)\idiff y\geq \frac{1}{\pi}\int_0^{t_0x}(m_1(y)-5(t_0x)^{-\delta_0})_+\phi(y/x)\idiff y =: \Phi(x).
\end{align*}
The first inequality above owes to the non-increasing property of $f$ on $[0,+\infty)$, and the second inequality follows from $m_1(x)\leq f(x)\leq 5x^{-\delta_0}$. Note that for $0\leq y\leq t_0x$,
\[(m_1(y)-5(t_0x)^{-\delta_0})_+\phi(y/x) \leq 2m_1(y).\]
Hence, we can use the dominated convergence theorem to obtain 
\[\lim_{x\rightarrow +\infty}\Phi(x) = \frac{2}{\pi}\int_0^{+\infty}m_1(y)\idiff y = b(m_1).\]
Recall $d(m_1)>1$ by Lemma \ref{lem:df_bound}. The above limit implies that, there exist some $L_1>0$ such that 
\[-\mtx{T}(f)(L_1) \geq \Phi(L_1)\geq \frac{d(m_1)}{2d(m_1)-1}\cdot b(m_1),\]
and hence,
\[\mtx{G}(f)(L_1) = 1 - \frac{\mtx{T}(f)(L_1)}{c(f)}\geq 1 + \frac{d(m_1)}{2d(m_1)-1}\cdot \frac{b(m_1)}{c(m_1)} = 1 + d(m_1).\]
The lemma is thus proved.
\end{proof}

We then derive a stronger uniform decay bound for $\mtx{R}(f)$ by again applying Lemma \ref{lem:g_L}. 

\begin{corollary}\label{cor:R_decay_strong}
For any $f\in \mathbb{D}$ and for all $x\geq 0$,
\[\exp\left(\int_0^x\frac{1-\mtx{G}(f)(y)}{y\mtx{G}(f)(y)}\idiff y\right)\leq \left(\frac{x}{L_1}\right)^{-(1+\delta_1)/2},\]
\[\mtx{M}(f)(x) \leq \left(\frac{x}{L_1}\right)^{-1-\delta_1},\]
and
\[\mtx{R}(f)(x)\leq \left(1 + \frac{3}{\delta_1}\right)\left(\frac{x}{L_1}\right)^{-1-\delta_1},\]
where 
\[\delta_1 := \frac{d(m_1)-1}{4d(m_1)}>0.\]
\end{corollary}

\begin{proof}
Write $d=d(f)$ and $d_1 = d(m_1)$. Since $d_1>1$ (Lemma \ref{lem:df_bound}), we have 
\[d_1 - \frac{5d_1-1}{3d_1+1} = \frac{(3d_1-1)(d_1-1)}{3d_1+1}>0.\]
Then, by Lemma \ref{lem:L_1}, we have 
\[\mtx{G}(f)(L_1)\geq 1 + d_1 \geq 1 + \frac{5d_1-1}{3d_1+1} =: 1+ a_1,\]
where 
\[1<a_1 = \frac{5d_1-1}{3d_1+1} \leq \frac{5d-1}{3d+1}.\]
Therefore, we can apply Lemma \ref{lem:g_L} with $L=L_1$ and $a=a_1>1$ to obtain 
\[\exp\left(\int_0^x\frac{1-\mtx{G}(f)(y)}{y\mtx{G}(f)(y)}\idiff y\right)\leq \left(\frac{x}{L}\right)^{-a_1/(1+a_1)}= \left(\frac{x}{L_1}\right)^{-(1+\delta_1)/2},\]
\[\mtx{M}(f)(x) \leq \exp\left(2\int_0^x\frac{1-\mtx{G}(f)(y)}{y\mtx{G}(f)(y)}\idiff y\right)\leq \left(\frac{x}{L_1}\right)^{-1-\delta_1},\]
and
\[\mtx{R}(f)(x)\leq \left(1+\frac{12d}{d-1}\right)\cdot\left(\frac{x}{L}\right)^{-2a_1/(1+a_1)}\leq \left(1+\frac{12d_1}{d_1-1}\right)\cdot\left(\frac{x}{L}\right)^{-2a_1/(1+a_1)} = \left(1 + \frac{3}{\delta_1}\right)\left(\frac{x}{L_1}\right)^{-1-\delta_1},\]
as desired.
\end{proof}

At this point, we are ready to conclude that the set $\mathbb{D}$ is nonempty and is closed under $\mtx{R}$.

\begin{theorem}\label{thm:R_close}
$\mathbb{D}$ is nonempty, and $\mtx{R}$ maps $\mathbb{D}$ into itself.
\end{theorem}

\begin{proof} The claim that $\mtx{R}$ maps $\mathbb{D}$ into itself follows from Corollaries \ref{cor:R_property}, \ref{cor:R_decay_weak}, \ref{cor:R_decay_strong}, \ref{cor:R_M_lower_bound}, and \ref{cor:R_1}. In order to prove $\mathbb{D}$ is nonempty, it suffices to show that $m_1\in \mathbb{D}$. Recall the definition \eqref{eqt:m1_definition} of $m_1$. Firstly, it is not hard to check that the function $g_1$ satisfies $g_1(x)$ is non-decreasing on $[0,\infty)$ and $g_1(\sqrt{s})$ is concave in $s$, which implies $m_1(x)$ is non-increasing on $[0,\infty)$ and $m_1(\sqrt{s})$ is convex in $s$ (by a similar argument as in the proof of Corollary \ref{cor:M_property}). Secondly, a direct derivative calculation shows that $m_1'(1)\leq -\eta$ (for that $\eta$ is very small). Thirdly, through the proofs of preceding lemmas and corollaries, it is straightforward to check that $m_1(x)\leq \min\big\{(1+3/\delta_1)(|x|/L_1)^{-1-\delta_1}\,,\, 5|x|^{-\delta_0}\big\}$. The above together imply that $m_1\in \mathbb{D}$, and the theorem is thus proved.
\end{proof}

\subsection{Continuity} In order to apply the Schauder fixed-point theorem, we also need the continuity of $\mtx{R}$ on $\mathbb{D}$ in the $L^\infty_\rho$ topology. This will rely on the continuity of $\mtx{T}$ and Lemma \ref{lem:df_continuous}.

\begin{theorem}\label{thm:R_continuous}
$\mtx{R}:\mathbb{D}\mapsto \mathbb{D}$ is continuous with respect to the $L^\infty_\rho$-norm.
\end{theorem}

\begin{proof}
Recall that $\rho(x) = (1+|x|)^{1+\delta_\rho}$ with $\delta_\rho = \delta_1/2$. Let $f,\bar f\in \mathbb{D}$ be arbitrary, and write $r = \mtx{R}(f)$, $m = \mtx{M}(f)$, $g = \mtx{G}(f)$, $d = d(f)$, $\bar r = \mtx{R}(\bar f)$, $\bar m = \mtx{M}(\bar f)$, $\bar g = \mtx{G}(\bar f)$, $\bar d = d(\bar f)$. Suppose that $\delta:=\|f-\bar f\|_{L^\infty_\rho} = \|\rho(f-\bar f)\|_{L^\infty}$ is sufficiently small. We will show that $\|\rho(r-\bar r)\|_{L^\infty}\lesssim \|\rho(f-\bar f)\|_{L^\infty}^{1/2}$, where the symbol ``$\lesssim$'' only hides a constant that does not depend on $f,\bar f$.

For any $x\geq 0$, we have
\begin{align*}
|\mtx{T}(f)(x)-\mtx{T}(\bar f)(x)| &= \frac{1}{\pi}\left|\int_0^{+\infty}(f(y)-\bar f(y))\left(\frac{y}{x}\ln\left|\frac{x+y}{x-y}\right|-2\right)\idiff y\right|\\
&\leq \frac{\delta}{\pi}\int_0^{+\infty}(1+y)^{-1-\delta_\rho}\left|\frac{y}{x}\ln\left|\frac{x+y}{x-y}\right|-2\right|\idiff y\\
&=\frac{\delta}{\pi}\cdot x\int_0^{+\infty}(1+tx)^{-1-\delta_\rho}\left|t\ln\left|\frac{t+1}{t-1}\right|-2\right|\idiff t\\
&= \frac{\delta}{\pi}\cdot x\int_0^{+\infty}(1+tx)^{-1-\delta_\rho}|\phi(t)|\idiff t,
\end{align*}
where $\phi$ is defined as in \eqref{eqt:phi}. It is not hard to show that there is some $t_0\in(0,1)$ such that $\phi(t_0)=0$, $0<\phi(t)\leq 2$ for $t\in[0,t_0)$, and $\phi(t)<0$ for $t>t_0$. Also note that $\phi\in L^1([0,+\infty))$. We then decompose and estimate the last integral of $t$ above as 
\begin{align*}
\int_0^{+\infty}(1+tx)^{-1-\delta_\rho}|\phi(t)|\idiff t &= \int_0^{t_0}(1+tx)^{-1-\delta_\rho}|\phi(t)|\idiff t + \int_{t_0}^{+\infty}(1+tx)^{-1-\delta_\rho}|\phi(t)|\idiff t\\
&\leq 2\int_0^{t_0}(1+tx)^{-1-\delta_\rho}\idiff t + (1+t_0x)^{-1-\delta_\rho}\int_{t_0}^{+\infty}|\phi(t)|\idiff t\\
&= \frac{2}{x\delta_\rho}\left(1 - (1+t_0x)^{-\delta_\rho}\right) + (1+t_0x)^{-1-\delta_\rho}\int_0^{+\infty}|\phi(t)|\idiff t \\
&\lesssim \min\{1\,,\,x^{-1}\}.
\end{align*}
We then obtain
\[|\mtx{T}(f)(x)-\mtx{T}(\bar f)(x)| \lesssim \delta\min\{ x\,,\,1\} . \]
A similar argument shows that $|\mtx{T}(f)(x)|,|\mtx{T}(\bar f)(x)|\lesssim \min\{x,1\}$. Combining these estimates with Lemma \ref{lem:df_bound} and Lemma \ref{lem:df_continuous} yields 
\begin{equation}\label{eqt:continuous_step1}
|g(x) - \bar g(x)| \lesssim \delta\min\{ x\,,\,1\}  + \delta^{1/2}\min\{ x\,,\,1\} \lesssim \delta^{1/2}\min\{ x\,,\,1\}.
\end{equation}
It then follows that, for any $x\geq 0$,  
\begin{align*}
\int_0^x\left|\frac{1-g(y)}{yg(y)}- \frac{1- \bar g(y)}{y\bar g(y)}\right|\idiff y &= \int_0^x\frac{|g(y)-\bar g(y)|}{yg(y)\bar g(y)}\idiff y\lesssim \delta^{1/2} \min\{x\,,\,|\ln x|\}.
\end{align*}
Write 
\[\psi(x) = \exp\left(\int_0^x\frac{1-g(y)}{yg(y)}\idiff y\right),\quad \bar \psi(x) = \exp\left(\int_0^x\frac{1-\bar g(y)}{y\bar g(y)}\idiff y\right).\]
Using Corollary \ref{cor:R_decay_strong}, we find
\begin{align*}
|\psi(x)^2 - \bar \psi(x)^2| &\lesssim (\psi(x)^2 + \bar\psi(x)^2) \int_0^x\left|\frac{1-g(y)}{yg(y)}- \frac{1- \bar g(y)}{y\bar g(y)}\right|\idiff y \\
&\lesssim \min\{1\,,\, x^{-1-\delta_1}\}\cdot \delta^{1/2} \min\{x\,,\,|\ln x|\}\\
&\lesssim \delta^{1/2}\min\{ x\,,\,x^{-1-\delta_\rho} \}.
\end{align*}
We have used the fact $|\econst^t-\econst^s|\leq |t-s|(\econst^t+\econst^s)$. This also implies 
\begin{equation}\label{eqt:continuous_step2}
|m(x)-\bar m(x)| = \left|\frac{\psi(x)^2}{g(x)} - \frac{\bar\psi(x)^2}{\bar g(x)}\right| \lesssim \delta^{1/2}\min\{x\,,\,x^{-1-\delta_\rho} \}.
\end{equation}

Next, we use \eqref{eqt:r_derivative} to obtain 
\begin{align*}
r'-\bar r' &= \frac{dm-(d+g-1)r}{xg} - \frac{\bar d\bar m-(\bar d+\bar g-1)\bar r}{x\bar g}\\
&= \left(\frac{1}{g}-\frac{1}{\bar g}\right)\frac{dm-(d+g-1)r}{x} \\
&\qquad + \frac{1}{x\bar g}\big((d-\bar d)(m-r) + \bar d(m-\bar m) - \bar d(r-\bar r) - (g-\bar g)r - (\bar g - 1)(r-\bar r) \big).
\end{align*}
Let $h = r-\bar r$. We can rearrange the display above to obtain
\begin{align*}
h' + \frac{\bar d + \bar g - 1}{x\bar g} \cdot h &= \left(\frac{1}{g}-\frac{1}{\bar g}\right)\frac{dm-(d+g-1)r}{x} + \frac{1}{x\bar g}\big((d-\bar d)(m-r) + \bar d(m-\bar m) - (g-\bar g)r\big)\\
&=: J(x).
\end{align*}
We multiply both sides of the equation above by $x^{\bar d}$ to obtain
\[(x^{\bar d}h)' - (\bar d-1)\frac{\bar g-1}{\bar g}x^{\bar d-1}h = x^{\bar d}J.\]
Since $\bar g \geq 1$ and $\bar d>1$, we have $(\bar d-1)(\bar g-1)/\bar g\geq 0$ for $x\geq0$. It then follows that, for $x\geq 0$,
\begin{align*}
x^{\bar d}|J| &= \left|(x^{\bar d}h)' - (\bar d-1)\frac{\bar g-1}{\bar g}x^{\bar d-1}h\right| \geq \left|(x^{\bar d}h)'\right| - (\bar d-1)\frac{\bar g-1}{\bar g}x^{\bar d-1}|h| \\
&\geq (x^{\bar d}|h|)' - (\bar d-1)\frac{\bar g-1}{\bar g}x^{\bar d-1}|h| \geq (x^{\bar d}|h|)' - (\bar d-1)\frac{2\bar d-1}{2\bar d }x^{\bar d-1}|h|.
\end{align*}
We have used $\bar g(x)\leq \bar g(+\infty) = 2\bar d$ for the last inequality above. Note that the derivative $|h|'$ is legit in weak sense. We have now reached 
\[(x^{\bar d}|h|)' + (\bar p-\bar d)x^{\bar d-1}|h| \leq x^{\bar d}|J|,\]
where 
\[\bar p = \bar d - \frac{(\bar d-1)(2\bar d-1)}{2\bar d} = \frac{3\bar d-1}{2\bar d} \in (1,\bar d).\]
It follows that
\[(x^{\bar p}|h(x)|)'  \leq  x^{\bar p}|J(x)|,\]
which leads to
\[|h(x)| \leq x^{-\bar p} \int_0^xy^{\bar p}|J(y)|\idiff y.\]
We then need to estimate $|J(x)|$. Note that, by Corollaries \ref{cor:R_decay_strong} and \ref{cor:R_M_lower_bound}, 
\[\frac{|r-m|}{x} \lesssim \min\{1\,,\, x^{-2-\delta_1}\}\lesssim \min\{1\,,\, x^{-2-\delta_\rho}\}.\]
Using this, Lemma \ref{lem:df_bound}, Lemma \ref{lem:df_continuous}, Corollary \ref{cor:R_decay_strong}, and the preceding estimates \eqref{eqt:continuous_step1} and \eqref{eqt:continuous_step2}, we can bound $J(x)$ as 
\[|J(x)| \lesssim \delta^{1/2}\min\{1\,,\,x^{-2-\delta_\rho} \}.\] 
Moreover, since $\bar f\in \mathbb{D}$, $\bar d \geq d(m_1)>1$, and thus
\[\bar p = \frac{3\bar d-1}{2\bar d} \geq \frac{5\bar d-1}{4\bar d}\geq \frac{5d_1-1}{4d_1} = 1 + \frac{d_1-1}{4d_1} = 1+\delta_1 > 1 +\delta_\rho.\]
Finally, we have
\begin{align*}
|r(x)-\bar r(x)|= |h(x)| &\lesssim \delta^{1/2} x^{-\bar p} \int_0^xy^{\bar p}\min\{1\,,\,y^{-2-\delta_\rho} \}\idiff y\\
&\leq \frac{\delta^{1/2}}{\bar p -1 -\delta_\rho}\min\{ x\,,\, x^{-1-\delta_\rho}\} \lesssim \delta^{1/2}(1+x)^{-1-\delta_\rho}.
\end{align*}
That is, for all $x\geq 0$,
\[\rho(x)|r(x) - \bar r(x)| = (1+x)^{1+\delta_\rho}|r(x) - \bar r(x)|\lesssim \delta^{1/2}.\]
Therefore, $\|r-\bar r\|_{L^\infty_\rho}\lesssim \delta^{1/2} = \|f-\bar f\|_{L^\infty_\rho}^{1/2}.$ This proves the continuity of $\mtx{R}:\mathbb{D}\mapsto\mathbb{D}$ in the $L^\infty_\rho$-norm.
\end{proof}

\subsection{Existence of a fixed point} One last ingredient for establishing existence of a fixed point of $\mtx{R}$ is the compactness of $\mathbb{D}$.

\begin{lemma}\label{lem:compactness} 
The set $\mathbb{D}$ is compact with respect to the $L_\rho^\infty$-norm.
\end{lemma}

\begin{proof} For any $f\in \mathbb{D}$, we use convexity and monotonicity to obtain
\[-\frac{f'(x)}{2x}\leq \frac{f(0)-f(x)}{x^2} \leq \min\{ 1\ , \frac{1}{x^2}\} ,\quad x>0.\]
implying that $|f'(x)|\leq \min\{2x,2x^{-1}\}\leq 2$. Based on this, we show that $\mathbb{D}$ is sequentially compact. 

Recall $\rho(x) = (1+|x|)^{1+\delta_\rho} = (1+|x|)^{1+\delta_1/2}$. Let $\{f_n\}_{n=1}^{+\infty}$ be an arbitrary sequence in $\mathbb{D}$. Initialize $n_{0,k}=k$, $k\geq 1$. For each integer $m\geq 1$, let $\epsilon_m = 2^{-m}$ and $X_m = C_1\epsilon_m^{-2/\delta_1}$ for some absolute constant $C_1>1$ that only depends on $\delta_1$ and $L_1$. We can choose $C_1$ so that, for all $n\geq 1$ and for all $x\geq X_m$, 
\[\rho(x)f_n(x)\leq (1+3\delta_1^{-1})(x/L_1)^{-1-\delta_1}\cdot (1+x)^{1+\delta_1/2}\leq \epsilon_m\] 
Furthermore, since $|f_n'(x)|\leq 2$ on $[0,X_m]$, we can apply Ascoli's theorem to select a sub-sequence $\{f_{n_{m,k}}\}_{k=1}^{+\infty}$ of $\{f_{n_{m-1,k}}\}_{k=1}^{+\infty}$ such that $\|\rho(f_{n_{m,i}}-f_{n_{m,j}})\|_{L^\infty}\leq 2\epsilon_m$ for any $i,j\geq 1$. Then the diagonal sub-sequence $\{f_{n_{m,m}}\}_{m=1}^{+\infty}$ is a Cauchy sequence in the $L_\rho^\infty$-norm. This proves that $\mathbb{D}$ is sequentially compact.
\end{proof}

We are finally ready to prove the existence of a fixed point of $\mtx{R}$ in $\mathbb{D}$.

\begin{theorem}\label{thm:existence_fixed_point}
The map $\mtx{R}$ has a fixed point $f\in\mathbb{D}$, i.e. $\mtx{R}(f)=f$.
\end{theorem}

\begin{proof}
By Theorem \ref{thm:R_continuous} and Lemma \ref{lem:compactness}, the nonempty set $\mathbb{D}$ is convex, closed and compact in the $L_\rho^\infty$-norm, and $\mtx{R}$ continuously maps $\mathbb{D}$ into itself. The Schauder fixed-point theorem (Fact \ref{fact:Schauder}) then guarantees that $\mtx{R}$ has a fixed point in $\mathbb{D}$. 
\end{proof}

\section{Properties of a fixed-point solution}\label{sec:properties}
In this section, we derive finer properties such as smoothness and far-field behavior of a fixed point $f = \mtx{R}(f)$ in $\mathbb{D}$ by using the fixed-point relation. We remark that all properties to be established below apply to any possible fixed point of $\mtx{R}$ in $\mathbb{D}$, though the uniqueness of a fixed point of $\mtx{R}$ is conjectured. 

As usual, we write $g=\mtx{G}(f), m=\mtx{M}(f)$ and $d=d(f)$. Provided that $f=\mtx{R}(f)$, we have
\begin{equation}\label{eqt:fixed_point_relation}
\begin{split}
&xgf' = (1-g)f - df + dm,\\
&xgm' = 2(1-g)m - xg'm.
\end{split}
\end{equation}

\subsection{Regularity} We first show that a fixed point $f=\mtx{R}(f)$ is actually infinitely smooth on $\R$, using the fact that the map $\mtx{T}(f)$ gains regularity by integration.

\begin{lemma}\label{lem:regularity_induction}
Given $f\in \mathbb{D}$, suppose that $b(f)<+\infty$. If $f\in H^p(\R)$ for some integer $p\geq 0$, then $\mtx{T}(f)'\in H^p(\R)$.
\end{lemma} 

\begin{proof}
In view of \eqref{eqt:T_to_laplacian}, we have 
\begin{equation}\label{eqt:Tf_Hf}
(x\mtx{T}(f))' = -\mtx{H}(xf) - b(f) = \mtx{H}(xf)(0)-\mtx{H}(xf) = x\cdot \frac{\mtx{H}(xf)(0)-\mtx{H}(xf)}{x} = -x\mtx{H}(f).
\end{equation}
We have used Lemma \ref{lem:Hilbert_property} for the last identity above. It follows that
\[\mtx{T}(f)(x) = -\frac{1}{x}\int_0^xy\mtx{H}(f)(y)\idiff y,\]
and thus
\[
\mtx{T}(f)'(x) = - \mtx{H}(f)(x) + \frac{1}{x^2}\int_0^xy\mtx{H}(f)(y)\idiff y = - \mtx{H}(f)(x) + \int_0^1t\mtx{H}(f)(tx)\idiff t.
\]
Then, for any integer $p\geq0$, we have
\[\mtx{T}(f)^{(p+1)}(x) = - \mtx{H}(f)^{(p)}(x) + \int_0^1t^{p+1}\mtx{H}(f)^{(p)}(tx)\idiff t,\]
which easily implies that 
\[\|\mtx{T}(f)\|_{\dot{H}^{p+1}(\R)} \leq C_p\|\mtx{H}(f)\|_{\dot{H}^p(\R)}= C_p\|f\|_{\dot{H}^p(\R)}.\]
This proves the lemma.
\end{proof}

We can then prove the smoothness of $f=\mtx{R}(f)$ by induction.

\begin{theorem}\label{thm:regularity}
Let $f\in \mathbb{D}$ be a fixed point of $\mtx{R}$. Then, $f,\mtx{M}(f),(xf)',(x\mtx{M}(f))'\in H^p(\R)$ for all $p\geq 0$.
\end{theorem}

\begin{proof}
Write $g = \mtx{G}(f)$, $m=\mtx{M}(f)$, and $d=d(f)$. Since $f\in \mathbb{D}$, we know $f\in L^2(\R)$. Moreover, by Corollary \ref{cor:R_decay_strong}, we also have $m\in L^2(\R)$. Now suppose that $f\in H^p(\R)$ for some $p\geq 0$. Lemma \ref{lem:regularity_induction} then implies that $g' \in H^p(\R)$. Recall that 
\begin{equation}\label{eqt:regularity_step}
m'(x) = -\left(g'(x) + 2\frac{g(x)-1}{x}\right)\frac{m(x)}{g(x)}.
\end{equation}
Also recall that $g(x)\geq g(0)=1$ and $m(x)\leq m(0)=1$. Moreover, it is not hard to check that $\xi'\in H^p(\R)$ implies $(\xi-\xi(0))/x\in H^p(\R)$ for any suitable function $\xi$. We then immediately have $m'\in H^p(\R)$. Rearranging the first equation of \eqref{eqt:fixed_point_relation}, we obtain
\[f'(x) + d\, \frac{f(x)-1}{x}= \left(d\,\frac{m(x)-1}{x} + (d-1)\,\frac{g(x)-1}{x}f(x) -d\, \frac{g(x)-1}{x}\right)\frac{1}{g(x)}=:h(x).\]
Since $f,g',m'\in H^p(\R)$, we apparently have $h\in H^p(R)$. Multiplying the equation above by $x^d$ yields
\[\big(x^d(f(x)-1)\big)' = x^dh(x),\]
which implies 
\[\frac{f(x) -1}{x} = x^{-d-1}\int_0^xy^dh(y)\idiff y = \int_0^1t^dh(tx)\idiff t.\]
This means $(f(x)-1)/x \in H^p(\R)$. Then, since 
\[f'(x) = h(x) - d\, \frac{f(x)-1}{x},\]
we further have $f'\in H^p(\R)$. That is, $f,m\in H^{p+1}(\R)$. Therefore, we can use induction to show that $f,m\in H^p(\R)$ for all $p\geq 0$. Next, we can use \eqref{eqt:regularity_step}, the above results, and the fact $\|xm\|_{L^\infty(\R)}<+\infty$ to inductively show that $(xm)'\in H^p(\R)$  for all $p\geq 0$. Moreover, we can compute that 
\[(xf(x))' = d\,\frac{m(x)}{g(x)}-(d-1)\,\frac{f(x)}{g(x)},\]
which implies $(xf)'\in H^p(\R)$ for all $p\geq 0$. This completes the proof.
\end{proof}

\subsection{Asymptotic behavior} Next, we study the asymptotic behavior of $f =\mtx{R}(f)$ as $x\rightarrow +\infty$. As we will see, both $f$ and $\mtx{M}(f)$ actually decay algebraically in the far field with decay rates given explicitly in terms of $d(f)$. The next lemma explains how the asymptotic behavior of $f$ relates to that of $\mtx{T}(f)$.

\begin{lemma}\label{lem:f_to_g_decay}
Given $f\in \mathbb{D}$, if $C_\delta := \sup_{x\in \R}\ x^{1+\delta}f(x) <+\infty$ for some $\delta\in(0,2)$, then 
\[\sup_{x\in \R}\ x^{\delta}\big(\mtx{T}(f)(x)-\mtx{T}(f)(+\infty)\big)\lesssim C_\delta.\]
Moreover, if the limit $D_\delta := \lim_{x\rightarrow +\infty} x^{1+\delta}f(x)$ exists and is finite, then
\[\lim_{x\rightarrow +\infty} x^{\delta}\big(\mtx{T}(f)(x)-\mtx{T}(f)(+\infty)\big) = \frac{D_\delta}{\pi}\int_0^{+\infty}\frac{1}{t^\delta}\ln\left|\frac{t+1}{t-1}\right|\idiff y.\]
\end{lemma}

\begin{proof}
For any $x>0$, we calculate that 
\begin{align*}
\mtx{T}(f)(x) - \mtx{T}(f)(+\infty) &= \mtx{T}(f)(x) + b(f) = \frac{1}{\pi}\int_0^{+\infty}f(y)\cdot \frac{y}{x}\ln\left|\frac{x+y}{x-y}\right|\idiff y\\
&\leq \frac{C_\delta}{\pi}\cdot\frac{1}{x}\int_0^{+\infty}\frac{1}{y^\delta}\ln\left|\frac{x+y}{x-y}\right|\idiff y = \frac{C_\delta}{\pi}\cdot\frac{1}{x^\delta}\int_0^{+\infty}\frac{1}{t^\delta}\ln\left|\frac{t+1}{t-1}\right|\idiff t\lesssim \frac{C_\delta}{x^\delta}.
\end{align*}
We have used the fact that the non-negative function $\frac{1}{t^\delta}\ln\left|\frac{t+1}{t-1}\right|$ is integrable on $[0,+\infty)$ for any $\delta\in(0,2)$. This proves the first result.

To prove the second result, we write the integral as
\begin{align*}
x^{\delta}\big(\mtx{T}(f)(x) - \mtx{T}(f)(+\infty)\big) &= \frac{x^{\delta}}{\pi}\int_0^{+\infty}f(y)\cdot \frac{y}{x}\ln\left|\frac{x+y}{x-y}\right|\idiff y\\
&= \frac{x^{1+\delta}}{\pi}\int_0^{+\infty}f(tx)\cdot t\ln\left|\frac{t+1}{t-1}\right|\idiff t \\
&= \frac{1}{\pi}\int_0^{+\infty}(tx)^{1+\delta}f(tx)\cdot \frac{1}{t^{\delta}}\ln\left|\frac{t+1}{t-1}\right|\idiff t.
\end{align*}
Since the limit $D_\delta := \lim_{x\rightarrow +\infty} x^{1+\delta}f(x)$ exists and is finite, we know $C_\delta := \sup_{x\in \R}\ x^{1+\delta}f(x)<+\infty$ and $\lim_{x\rightarrow +\infty} (tx)^{1+\delta}f(tx) = D_\delta$ for any fixed $t>0$. Also note that the function $t^{-\delta}\ln\left|\frac{t+1}{t-1}\right|$ is absolutely integrable on $[0,+\infty)$ for any $\delta\in(0,2)$. We then use the dominated convergence theorem to obtain
\[\lim_{x\rightarrow +\infty}x^{\delta}\big(\mtx{T}(f)(x) - \mtx{T}(f)(+\infty)\big) = \frac{D_\delta}{\pi}\int_0^{+\infty}\frac{1}{t^\delta}\ln\left|\frac{t+1}{t-1}\right|\idiff y,\]
as claimed.
\end{proof}

We can now compute the decay rates of $\mtx{M}(f)$ and $\mtx{R}(f)$ by studying the asymptotic behavior of $\mtx{G}(f)$. 

\begin{theorem}\label{thm:r_asymptotic}
For any $f\in \mathbb{D}$, there are some finite constants $C_m,C_r>0$ such that
\[\lim_{x\rightarrow+\infty} x^{1+2\delta_d} \mtx{M}(f)(x) = C_m,\]
and
\[\lim_{x\rightarrow+\infty} x^{1+\delta_d} \mtx{R}(f)(x) = C_r,\]
where 
\begin{equation}\label{eqt:delta_d}
\delta_d = \frac{d(f)-1}{2d(f)}.
\end{equation}
As a consequence, if $f$ is a fixed point of $\mtx{R}$ in $\mathbb{D}$, then 
\[\lim_{x\rightarrow+\infty} x^{1+\delta_d} f(x) = C_r.\]
\end{theorem}

\begin{proof}
Let $r = \mtx{R}(f)$, $m=\mtx{M}(f)$, $g = \mtx{G}(f)$, and $d=d(f)$. We first show that there is some finite constant $C_0>0$ such that 
\begin{equation}\label{eqt:asymptotic_step1}
\lim_{x\rightarrow+\infty} x^{1/2+\delta_d}\exp\left(\int_0^x\frac{1-g(y)}{yg(y)}\idiff y\right) = C_0.
\end{equation}
Recall that $g(+\infty)=2d$. We can compute 
\begin{equation}\label{eqt:asymptotic_step2}
\begin{split}
\int_1^x\frac{1-g(y)}{yg(y)}\idiff y &= \int_1^x\frac{1}{y}\cdot \frac{g(+\infty)-g(y)}{g(+\infty) g(y)}\idiff y + \frac{1-g(+\infty)}{g(+\infty)}\cdot\int_1^x\frac{1}{y}\idiff y\\
&= \int_1^x\frac{1}{y}\cdot \frac{g(+\infty)-g(y)}{g(+\infty) g(y)}\idiff y - \left(\frac{1}{2} + \delta_d\right)\ln x.
\end{split}
\end{equation}
Since $f\lesssim \min\{1\,,\,x^{-1-\delta_1}\}$, we have by Lemma \ref{lem:f_to_g_decay} that 
\[g(+\infty) - g(x) = \frac{\mtx{T}(f)(x) - \mtx{T}(f)(+\infty)}{c(f)}\lesssim x^{-\delta_1}.\]
Hence, the first term in the last line of \eqref{eqt:asymptotic_step2} is finite as $x\rightarrow +\infty$, that is,
\[\int_1^{+\infty}\frac{1}{y}\cdot \frac{g(+\infty)-g(y)}{g(+\infty) g(y)}\idiff y<+\infty.\]
We then use \eqref{eqt:asymptotic_step2} to obtain 
\begin{align*}
\lim_{x\rightarrow+\infty} x^{1/2+\delta_d} \exp\left(\int_0^x\frac{1-g(y)}{yg(y)}\idiff y\right) &= \lim_{x\rightarrow+\infty} \exp\left(\int_0^1\frac{1-g(y)}{yg(y)}\idiff y + \int_1^x \frac{g(+\infty)-g(y)}{yg(+\infty) g(y)}\idiff y\right)\\
&=\exp\left(\int_0^1\frac{1-g(y)}{yg(y)}\idiff y + \int_1^{+\infty}\frac{g(+\infty)-g(y)}{yg(+\infty) g(y)}\idiff y\right)\\
&=: C_0,
\end{align*}
as desired. It immediately follows that 
\[\lim_{x\rightarrow+\infty} x^{1+2\delta_{d}}m(x) = \lim_{x\rightarrow+\infty} x^{1+2\delta_{d}}\frac{1}{g(x)}\exp\left(2\int_0^x\frac{1-g(y)}{yg(y)}\idiff y\right) = \frac{C_0^2}{2d} =: C_m.\]

Next, we prove the asymptotic behavior of $r$. Let $A(x),B(x)$ be defined as in \eqref{eqt:A_B}. Recall that $A(x)$ is non-decreasing on $[0,+\infty)$, and thus $A(+\infty)$ is well-defined. We then use \eqref{eqt:asymptotic_step1} to find that
\begin{align*}
A(+\infty) &= \int_0^{+\infty} dy^{d-1}\cdot\frac{1}{g(y)^2}\exp\left(\big(d+1\big)\int_0^y\frac{1-g(z)}{zg(z)}\idiff z\right)\idiff y\\
&\lesssim \int_0^{+\infty} y^{d-1} \min\{1\,,\,y^{-(1/2+\delta_d)(d+1)}\}\idiff y\\
&\lesssim \int_0^1 y^{d-1}\idiff y + \int_0^{+\infty} y^{-1-\delta_d}\idiff y<+\infty.
\end{align*}
Moreover, we can also use \eqref{eqt:asymptotic_step1} to obtain
\[\lim_{x\rightarrow+\infty} \frac{x^{1+\delta_d}}{B(x)} = \lim_{x\rightarrow+\infty} x^{-d+1+\delta_d}\exp\left((1-d)\int_0^x\frac{1-g(y)}{yg(y)}\idiff y\right) = \frac{1}{C_0^{d-1}}.\]
Therefore, 
\begin{align*}
\lim_{x\rightarrow+\infty} x^{1+\delta_{d}}r(x) &= \lim_{x\rightarrow+\infty} x^{1+\delta_{d}}\frac{A(x)}{B(x)} = \frac{A(+\infty)}{C_0^{d-1}} =: C_r.
\end{align*}
This completes the proof.
\end{proof}

\subsection{Estimates of $d(f)$} We have shown that the asymptotic decay rates of $f$ and $\mtx{M}(f)$ are given in terms of $d(f)$, or equivalently, in terms of the ratio $b(f)/c(f)$. What is left undone is to estimate the value of $d(f)$ for a fixed point $f=\mtx{R}(f)$. We first prove a useful identity as follows.

\begin{lemma}\label{lem:b_c_identity}
Let $f\in \mathbb{D}$ be a fixed point of $\mtx{R}$. Then, 
\[(b(f)-c(f))b(f) - (b(f) + c(f))b(m) =2Q(f),\]
where
\[Q(f) := -\frac{2}{\pi}\int_0^{+\infty}xf(x)\mtx{T}(f)'(x)\idiff x \geq 0.\]
\end{lemma}

\begin{proof}
Rearranging the first equation of \eqref{eqt:fixed_point_relation} we get
\[df + xf' - dm = (1-g)(xf)'.\]
Multiply both sides of the equation above by $2/\pi$ and then integrating them over $[0,+\infty)$ yields 
\[db(f) + \frac{2}{\pi}\int_0^{+\infty}xf'(x)\idiff x - db(m) = \frac{2}{\pi}\int_0^{+\infty}(1-g(x))(xf(x))'\idiff x.\]
Since $f\lesssim x^{-1-\delta_{\rho}}$, we can use integration by parts to obtain
\[ \frac{2}{\pi}\int_0^{+\infty}xf'(x)\idiff x = -\frac{2}{\pi}\int_0^{+\infty}f(x)\idiff x = -b(f), \]
and
\begin{align*}
\frac{2}{\pi}\int_0^{+\infty}(1-g(x))(xf(x))'\idiff x &= \frac{2}{\pi}\int_0^{+\infty}g'(x)\cdot xf(x)\idiff x \\
&= \frac{1}{c(f)}\cdot \frac{2}{\pi}\int_0^{+\infty}(-\mtx{T}(f)'(x))\cdot xf(x)\idiff x = \frac{Q(f)}{c(f)}.
\end{align*}
Hence, we obtain
\[(d-1)b(f) - db(m) = \frac{Q(f)}{c(f)}.\]
Substituting $d = (b(f)+c(f))/2c(f)$ in this equation yields the desired identity.

Next, we show that $Q(f)\geq 0$. In view of \eqref{eqt:T_to_laplacian}, we can rewrite $Q(f)$ as
\begin{align*}
Q(f) &= -\frac{2}{\pi}\int_0^{+\infty}\left(\frac{1}{x}(-\Delta)^{-1/2}(xf)(x) - b(f)\right)'xf(x)\idiff x \\
&= \frac{2}{\pi}\int_0^{+\infty}\mtx{H}(xf)(x)\cdot f(x)\idiff x + \frac{2}{\pi}\int_0^{+\infty}\frac{(-\Delta)^{-1/2}(xf)(x)\cdot f(x)}{x}\idiff x\\
&= -\frac{b(f)^2}{2} + \frac{2}{\pi}\int_0^{+\infty}\big(\mtx{T}(f)(x)+b(f)\big)f(x)\idiff x\\
&= -\frac{b(f)^2}{2} + \frac{2}{\pi^2}\int_0^{+\infty}\int_0^{+\infty}f(x)f(y)\frac{y}{x}\ln\left|\frac{x+y}{x-y}\right|\idiff x\idiff y\\
&= -\frac{b(f)^2}{2} + \frac{1}{\pi^2}\int_0^{+\infty}\int_0^{+\infty}f(x)f(y)\left(\frac{x}{y}+\frac{y}{x}\right)\ln\left|\frac{x+y}{x-y}\right|\idiff x\idiff y\\
&= \frac{1}{\pi^2}\int_0^{+\infty}\int_0^{+\infty}f(x)f(y)\left(\left(\frac{x}{y}+\frac{y}{x}\right)\ln\left|\frac{x+y}{x-y}\right|-2\right)\idiff x\idiff y.
\end{align*}
In the third line above, we have used Lemma \ref{lem:Hilbert_property} to compute that 
\[\frac{2}{\pi}\int_0^{+\infty}\mtx{H}(xf)(x)\cdot f(x)\idiff x = \frac{1}{\pi}\int_0^{+\infty}\frac{\mtx{H}(xf)(x)\cdot xf(x)}{x}\idiff x = -\frac{1}{2}\left(\mtx{H}(fx)(0)\right)^2= -\frac{b(f)^2}{2}.\]
Note that 
\[
\left(\frac{x}{y}+\frac{y}{x}\right)\ln\left|\frac{x+y}{x-y}\right|-2\geq 0,\quad \text{for all $x,y\geq 0$},
\]
Therefore, for all $f\in \mathbb{D}$,
\begin{equation}\label{eqt:Q_formula}
Q(f) = \frac{1}{\pi^2}\int_0^{+\infty}\int_0^{+\infty}f(x)f(y)\left(\left(\frac{x}{y}+\frac{y}{x}\right)\ln\left|\frac{x+y}{x-y}\right|-2\right)\idiff x\idiff y\geq 0.
\end{equation}
This concludes the proof.
\end{proof}

Combining Lemma \ref{lem:b_c_identity} and a few estimates in Lemma \ref{lem:df_bound}, we can bound $d(f)$ from below as follows.

\begin{corollary}\label{cor:fixed-point_df_bound}
Let $f\in \mathbb{D}$ be a fixed point of $\mtx{R}$. Then, $b(f)/c(f)\geq 1 + \sqrt{10}/2$ and $d(f) \geq  1 + \sqrt{10}/4$. As a consequence, 
\[\delta_d = \frac{d(f)-1}{2d(f)} \geq \frac{\sqrt{10}}{8 + 2\sqrt{10}}\approx \frac{1}{4.5298}.\]
\end{corollary}

\begin{proof}
Let $m = \mtx{M}(f)$, $d= d(f)$, and $k = b(f)/c(f)$. From Lemma \ref{lem:b_c_identity} we find that 
\[(k-1)k - (k +1)\frac{b(m)}{c(f)} = \frac{2Q(f)}{c(f)^2}.\]
Since $m_0(x):=(1+x^2/2)^{-2}\leq m(x)\leq \mtx{R}(f)(x)=f(x)$ for all $x$, we have
\[b(m)\geq b(m_0)=\frac{\sqrt{2}}{2}=c(m_0) \geq c(f).\]
Moreover, we can use $f(x)\geq m_0(x)$ and the formula \eqref{eqt:Q_formula} to find that 
\[
Q(f)\geq Q(m_0) = -\frac{2}{\pi}\int_0^{+\infty}xm_0(x)\mtx{T}(m_0)'(x)\idiff x= \frac{16\sqrt{2}}{\pi}\int_0^{+\infty}\frac{x^2}{(2+x^2)^4}\idiff x = \frac{1}{8}.
\]
We have used the straightforward calculation that 
\[\mtx{T}(m_0)(x) = -\frac{\sqrt{2}}{2}\cdot \frac{x^2}{2+x^2}.\]
We thus obtain
\[(k-1)k - (k +1)\geq \frac{1}{2},\]
which implies $k\geq 1 + \sqrt{10}/2$. Hence, $d=(b(f)+c(f))/2c(f) = (1+k)/2\geq 1 + \sqrt{10}/4$.
\end{proof}

We conclude this section with a formal proof of our main theorem in the introduction.

\begin{proof}[Proof of Theorem \ref{thm:main}]
From Proposition \ref{prop:fixed_point_solution} and Theorem \ref{thm:existence_fixed_point} we know the self-similar profiles equations \eqref{eqt:main} admit a solution $(\om,v,c_l,c_\om)$ with $\om = xf, v = c_lx\mtx{M}(f)/2$, and 
\[c_l = b(f) + c(f),\quad c_\om = \frac{c(f)-b(f)}{2},\]
where $f\in \mathbb{D}$ is a fixed point of $\mtx{R}$. By Lemma \ref{lem:df_bound} we know $b(f)>c(f)$, and thus we have $c_\om <0$. Moreover, we can compute that 
\[-\frac{c_\om}{c_l} = \frac{b(f)-c(f)}{2(b(f)+c(f))} = \frac{d(f)-1}{2d(f)} = \delta_d.\]
Since $d(f)<+\infty$, we have $\delta_d<1/2$; and by Corollary \ref{cor:fixed-point_df_bound}, we know $\delta_d>1/4.53$.

In view of the scaling property \eqref{eqt:scaling}, we can renormalize the solution (by only tuning $\alpha$) so that $c_\om = -1$. Note that the ratio $c_\om/c_l$ is invariant under such rescaling. Hence, we have the estimate $c_l\in (2,4.53)$. Also note that the monotonicity and convexity properties of $\om/x$ and $v/x$ are invariant under renormalization. This proves the existence of an exact self-similar solution of the form \eqref{eqt:self-similar_solution} for the HL model \eqref{eqt:HL_model}, with the self-similar profiles satisfying properties (1) and (2) in Theorem \ref{thm:main}. Moreover, (3) the regularity and (4) the asymptotic decay rates of the profiles follow form Theorems \ref{thm:regularity} and \ref{thm:r_asymptotic}, respectively. This completes the proof.
\end{proof}

\section{Numerical verification}\label{sec:numerical}
In this final section, we perform a numerical study to verify and visualize our theoretical results. To this end, we need to obtain numerically accurate self-similar profiles of the HL model.

As a key step in their computer-assisted proof, Chen, Hou, and Huang \cite{chen2022asymptotically} obtained accurate approximate self-similar profiles by numerically solving the dynamic rescaling equations \eqref{eqt:dynamic_rescaling} of the HL model with high-order numerical schemes. Owing to the nonlinear stability under the normalization conditions \eqref{eqt:non-degeneracy_condition}, the numerical solution (for a large class of initial data) can easily converge to an approximate steady state with extremely small point-wise residual errors. Their computer codes and stored output data can be found in \cite{Matlabcode}. 

Instead of using the codes and data off-the-shelf, we construct our own approximate self-similar profiles by numerically solving the fixed-point problem $f=\mtx{R}(f)$ using a direct iterative method. That is, starting with some smooth initial function $f^{(0)}\in \mathbb{D}$, we numerically compute
\begin{equation}\label{eqt:iteration}
f^{(n+1)} = \mtx{R}(f^{(n)}),\quad n=0,1,2, ....
\end{equation}
More precisely, each iteration is computed in the order of \eqref{eqt:fixed-point_formulation}. After the scheme converges numerically, the self-similar profiles $\om,v$ and the scaling factors $c_l,c_\om$ can be recovered as in Proposition \ref{prop:fixed_point_solution} and renormalized as in \eqref{eqt:scaling}. Note that a similar fixed-point method for computing the self-similar profiles of the 1D gCLM model was proposed in \cite{huang2024self} by the same authors.

We perform the fixed-point computation for two purposes. The first one is to test the well-posedness of the fixed-point problem $f=\mtx{R}(f)$. We have not been able to prove the uniqueness of a fixed-point of $\mtx{R}$ nor the convergence of the scheme \eqref{eqt:iteration}. Nevertheless, this iterative method converges quickly for a bunch of arbitrarily picked initial data in $\mathbb{D}$ with the maximum residual $\|f^{(n)}-\mtx{R}(f^{(n)})\|_{L^\infty}$ dropped below an extremely small tolerance ($10^{-10}$ in our computations, the same standard as in \cite{chen2022asymptotically}) only within hundreds of iterations. For example, for the initial data $f^{(0)} = (1+x^2)^{-1}\in \mathbb{D}$, the residual $\|f^{(n)}-\mtx{R}(f^{(n)})\|_{L^\infty}$ decreases below $10^{-10}$ within 270 iterations; for the initial data $f^{(0)} = (1+x^2/2)^{-2}$ that is not in $\mathbb{D}$, the residual $\|f^{(n)}-\mtx{R}(f^{(n)})\|_{L^\infty}$ decreases below $10^{-10}$ within 250 iterations. This makes us believe that the fixed-point of $\mtx{R}$ in $\mathbb{D}$ is unique and the scheme \eqref{eqt:iteration} is convergent over $\mathbb{D}$ (and probably convergent for more general initial functions under weaker assumptions). 

The second purpose is to check whether the self-similar profiles determined by the fixed-point method in this paper and those obtained by solving the dynamic rescaling equations \eqref{eqt:dynamic_rescaling} as in \cite{chen2022asymptotically} are identical under proper rescaling. For the profiles constructed in \cite{chen2022asymptotically}, we obtain the profile data, denoted by $(\bar \omega, \bar v ,\bar u, \bar c_l ,\bar c_\om)$, from their open source code \cite{Matlabcode}. We then construct the corresponding functions $\bar f = \bar \om/x $, $\bar m = 2\bar v/(\bar c_l x)$, and $\bar g = (\bar c_l + \bar u/x)/(c_l + \bar u'(0))$ (as in \eqref{eqt:change_variable_1}) and properly rescale them (according to \eqref{eqt:scaling}) so that $\bar f(0) = 1$ and $\lim_{x\rightarrow 0}\bar g'(x)/x = \bar g''(0) =1$. Finally, we compare $(\bar f,\bar m)$ with our numerically obtained fixed-point solution $(f,m)$ where $m = \mtx{M}(f)$ (see Figure \ref{fig:comparison}). It turns out that the numerical profiles obtained by the iterative method \eqref{eqt:iteration} and those obtained by solving the dynamic rescaling equations under consistent normalization are almost identical only up to grid-point errors at the level of $10^{-7}$. The errors are likely due to the differences in the discretization methods. This convincingly supports our conjecture on uniqueness. Moreover, the ratio $c_l/|c_\om|$ numerically computed by our fixed-point iteration is approximately $2.99869$ (rounding to five decimal places), which is very close to the estimated value $\bar c_l/|\bar c_\om| = 2.99870$ given in \cite{chen2022asymptotically}. In fact, our result satisfies the rigorous computer-assisted estimate $|(c_l/|c_\om|) - 2.99870|\leq 6\times 10^{-5}$ established in \cite{chen2022asymptotically}.

\begin{figure}[!ht]
\centering
    \begin{subfigure}[b]{0.42\textwidth}
        \includegraphics[width=1\textwidth]{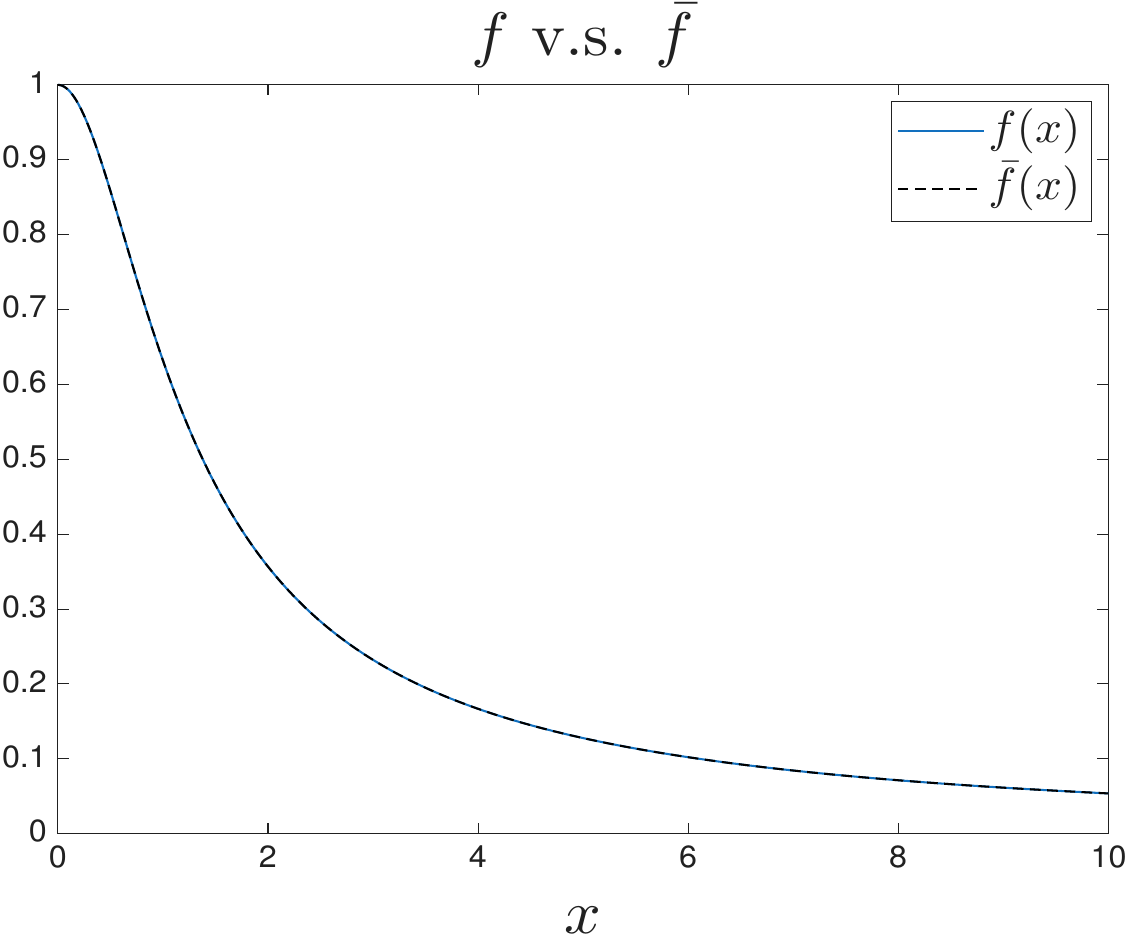}
        \caption{$f(x)$ versus $\bar f(x)$}
    \end{subfigure}\qquad
    \begin{subfigure}[b]{0.42\textwidth}
        \includegraphics[width=1\textwidth]{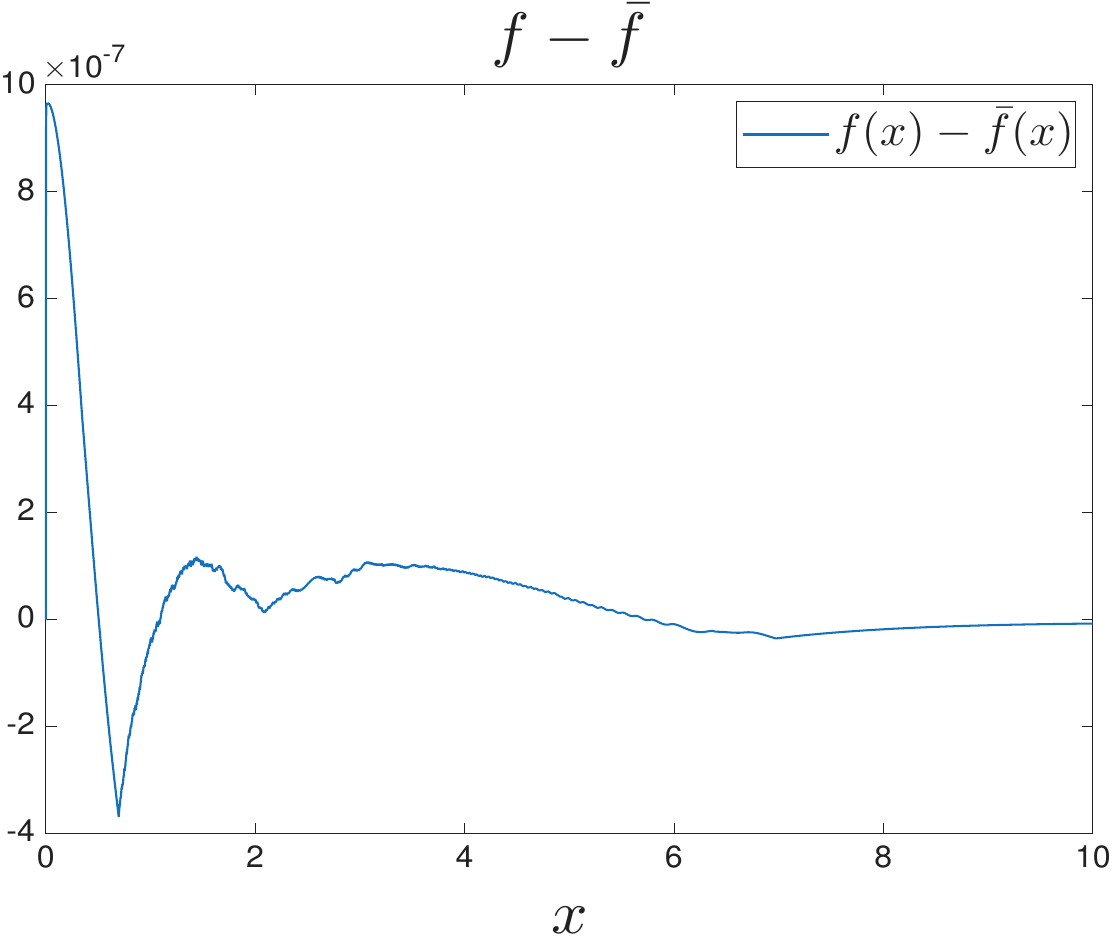}
        \caption{$f(x)-\bar f(x)$}
    \end{subfigure}
    \vspace{2mm}\\
    \begin{subfigure}[b]{0.42\textwidth}
        \includegraphics[width=1\textwidth]{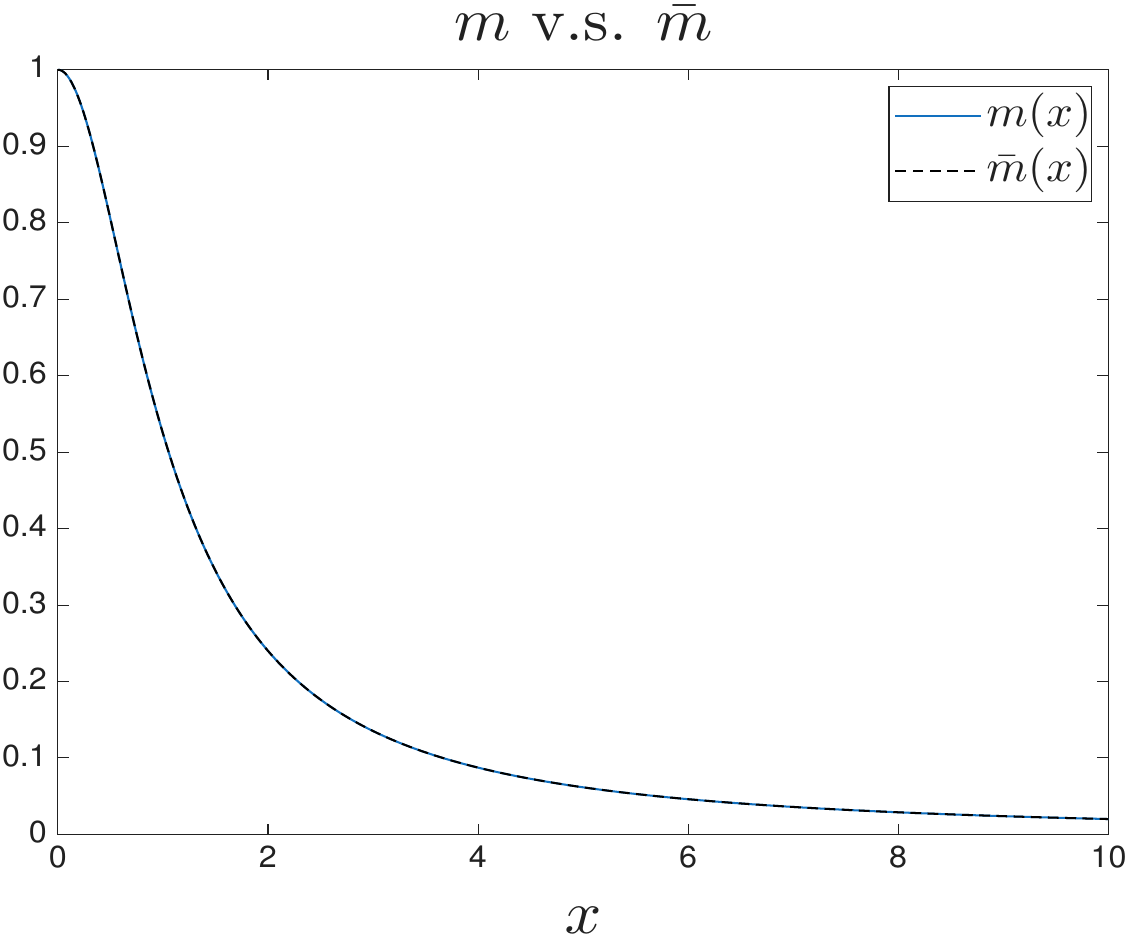}
        \caption{$m(x)$ versus $\bar m(x)$}
    \end{subfigure}\qquad
    \begin{subfigure}[b]{0.42\textwidth}
        \includegraphics[width=1\textwidth]{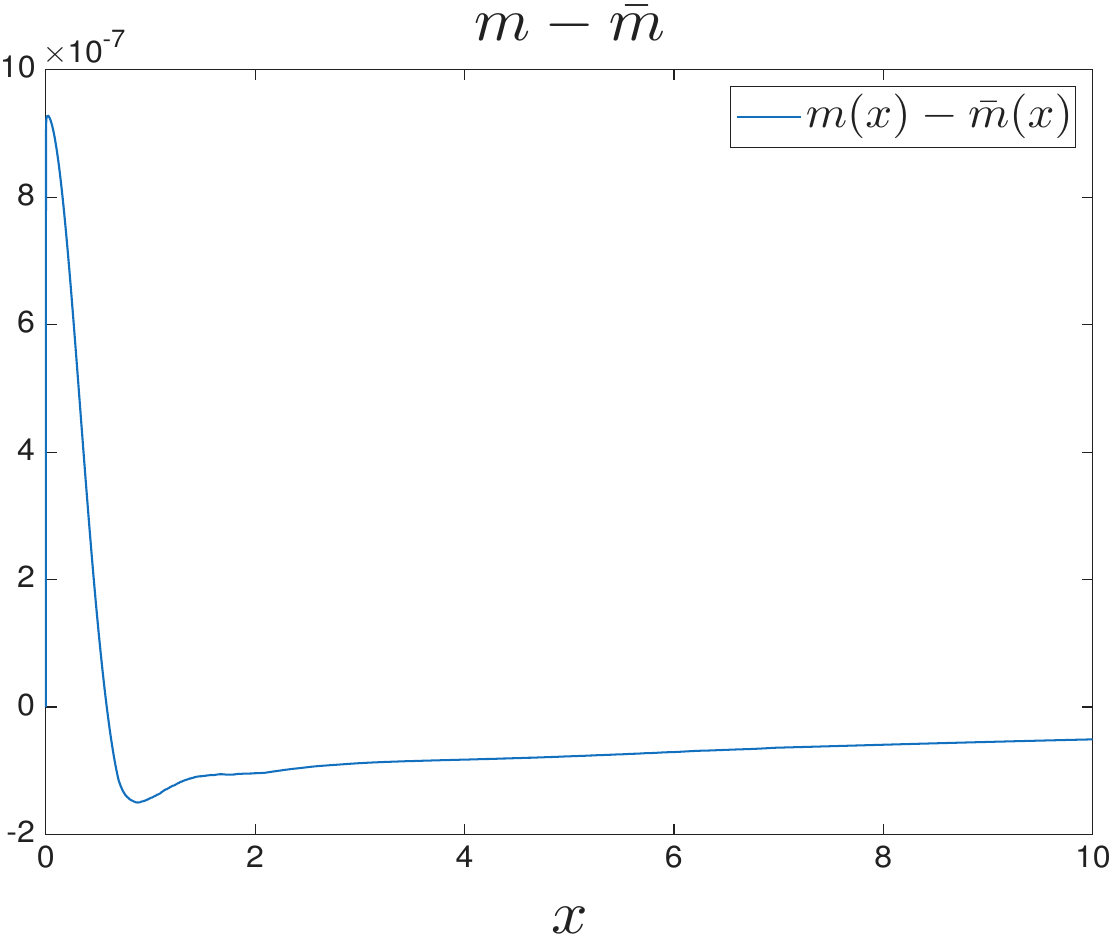}
        \caption{$m(x)-\bar m(x)$}
    \end{subfigure}
    \caption[Comparison]{Comparison between $(f,m)$ and $(\bar f,\bar m)$. Here, $f$ and $m = \mtx{M}(f)$ are obtained numerically using the fixed-point iteration scheme \eqref{eqt:iteration}; $\bar f$ and $\bar m$ are constructed from the data in \cite{Matlabcode} under proper normalization. The maximal difference between $f$ and $\bar f$ (or between $m$ and $\bar m$) over all grid points (including those beyond the plotting range in these figures) is below $10^{-6}.$}
    \label{fig:comparison}
\end{figure}

We remark that the most numerically expensive step in one iteration is to compute the function $\mtx{T}(f^{(n)})$, which involves the evaluation of a linear transform with a dense kernel at all mesh points. The mesh points are distributed adaptively over a sufficiently large one-sided interval $[0,10^{16}]$ (solutions are truncated to $0$ for $x\geq 10^{16}$). An analogous computation (recovering $u$ from $\om$) is also needed in every time step when solving the dynamic rescaling equations \eqref{eqt:dynamic_rescaling} numerically, and it takes more than tens of thousands of time steps for the solution to converge in time. Hence, our method is empirically more efficient than numerically solving the dynamic rescaling equations in obtaining approximate self-similar profiles. 

Finally, we provide some plots of the numerically constructed profiles to verify and visualize some of their theoretically proved properties. Figure \ref{fig:f} plots the numerically obtained fixed point $f$ and the corresponding $\mtx{M}(f)$ in coordinates $x$ and $s=x^2$ respectively, verifying that they are both monotone decreasing in $x$, convex in $s = x^2$, and lower bounded by $(1+x^2/2)^{-2}$ for $x\geq 0$. Figure \ref{fig:g} plots the corresponding $\mtx{G}(f)$ in a similar way, verifying that it is monotone increasing in $x$, concave in $x^2$, and upper bounded by $1+x^2/2$ for $x\geq 0$. Figure \ref{fig:asymptotic} demonstrates asymptotic decay rates of $f$ and $\mtx{M}(f)$ for sufficiently large $x$, verifying the statements in Theorem \ref{thm:r_asymptotic}.

\begin{figure}[!ht]
\centering
    \begin{subfigure}[b]{0.41\textwidth}
        \includegraphics[width=1\textwidth]{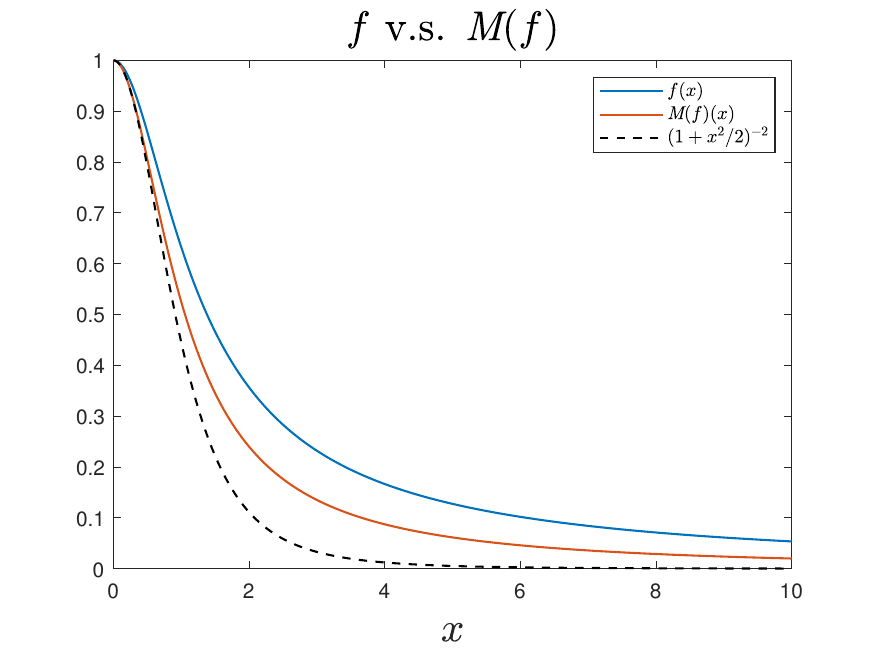}
        \caption{$f(x)$ and $\mtx{M}(f)(x)$}
    \end{subfigure}
    \begin{subfigure}[b]{0.41\textwidth}
        \includegraphics[width=1\textwidth]{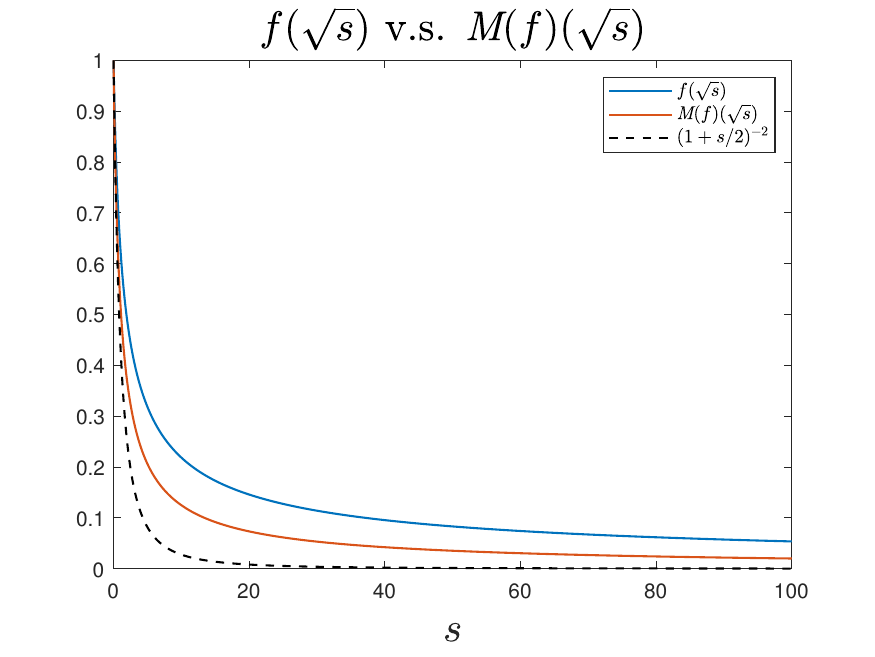}
        \caption{$f(\sqrt{s})$ and $\mtx{M}(f)(\sqrt{s})$}
    \end{subfigure}
    \caption[f]{The numerically obtained fixed point $f$ and the corresponding $\mtx{M}(f)$ plotted (a) in coordinate $x$ and (b) in coordinate $s=x^2$. The dashed line represents the lower bound $(1+x^2/2)^{-2}=(1+s/2)^{-2}$ for functions in $\mathbb{D}$.}
    \label{fig:f}
\end{figure}

\begin{figure}[!ht]
\centering
    \begin{subfigure}[b]{0.41\textwidth}
        \includegraphics[width=1\textwidth]{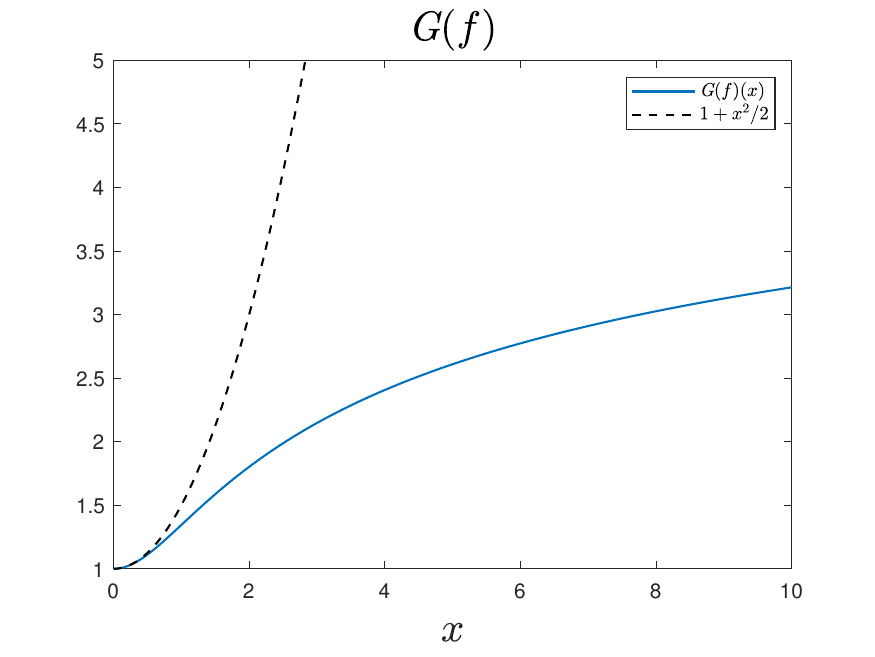}
        \caption{$\mtx{G}(f)(x)$}
    \end{subfigure}
    \begin{subfigure}[b]{0.41\textwidth}
        \includegraphics[width=1\textwidth]{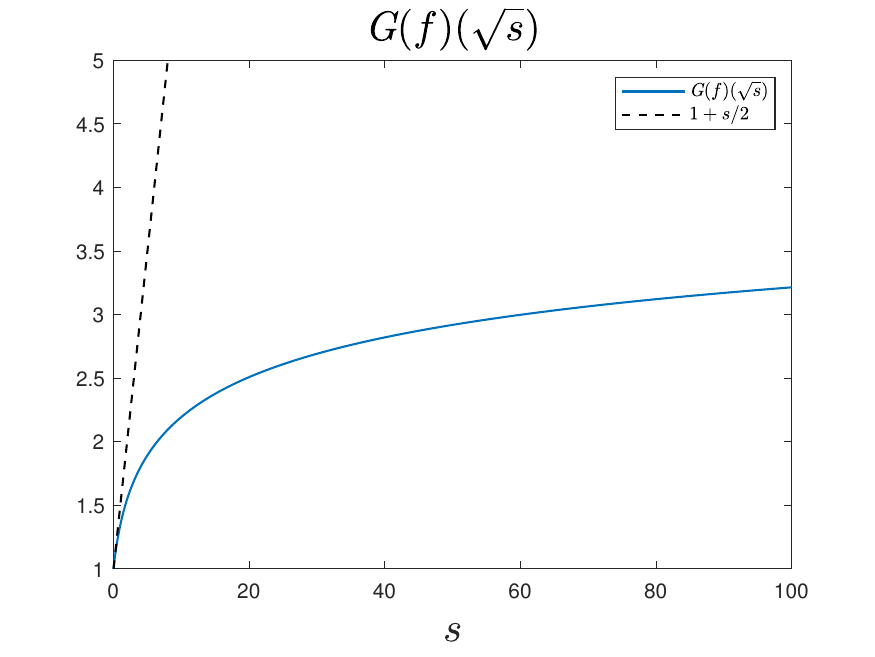}
        \caption{$\mtx{G}(f)(\sqrt{s})$}
    \end{subfigure}
    \caption[g]{$\mtx{G}(f)$ of the numerically obtained fixed point $f$ plotted (a) in coordinate $x$ and (b) in coordinate $s=x^2$. The dashed line represents the upper bound $1+x^2/2=1+s/2$.}
    \label{fig:g}
\end{figure}

\begin{figure}[!ht]
\centering
    \begin{subfigure}[b]{0.41\textwidth}
        \includegraphics[width=1\textwidth]{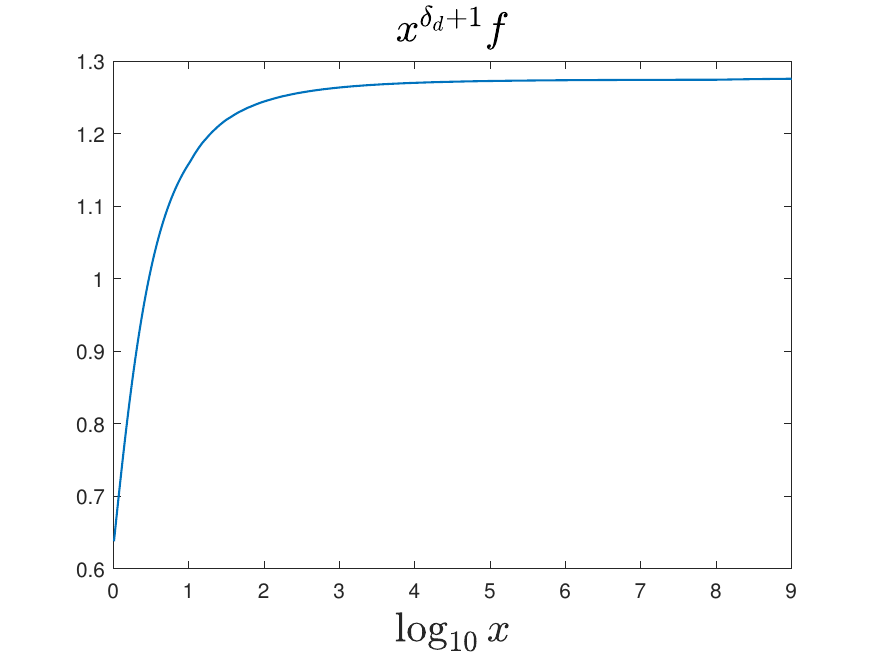}
        \caption{$x^{1+\delta_d}f(x)$}
    \end{subfigure}
    \begin{subfigure}[b]{0.41\textwidth}
        \includegraphics[width=1\textwidth]{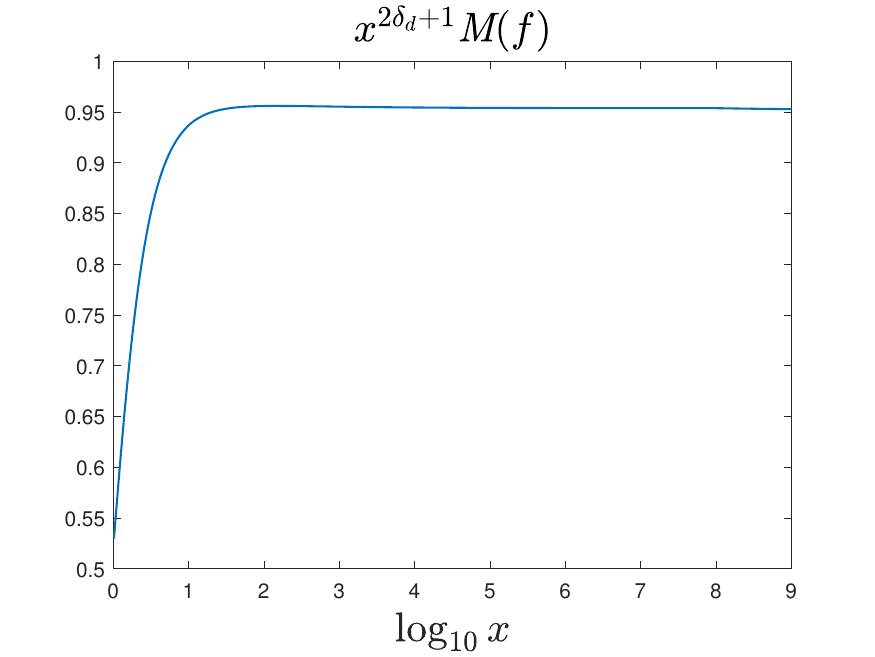}
        \caption{$x^{1+2\delta_d}\mtx{M}(f)(x)$}
    \end{subfigure}
    \caption[asymptotic]{Demonstrations of algebraic decay rates of (a) the numerically obtained fixed point $f$ and (b) the corresponding $\mtx{M}(f)$.}
    \label{fig:asymptotic}
\end{figure}

\appendix

\section{Useful facts}\label{sec:facts}

\subsection{Special function $F_1$}\label{sec:F_1}

We define
\begin{equation}\label{eqt:F1_definition}
F_1(t):= \frac{t^2-1}{2t}\ln\left|\frac{t+1}{t-1}\right| + 1, \quad t\geq 0.
\end{equation}
The derivative of $F$ reads
\[F_1'(t) = \frac{t^2+1}{2t^2}\ln\left|\frac{t+1}{t-1}\right| - \frac{1}{t}.\]
For $t\in[0,1)$, $F_1(t)$ and $F_1'(t)$ have the Taylor expansions 
\[F_1(t) = \suml_{n=1}^\infty\frac{2t^{2n}}{4n^2-1},\quad F_1'(t) = \suml_{n=1}^\infty\frac{4nt^{2n-1}}{4n^2-1}.\]
For $t\in[0,1)$, $F_1(1/t)$ and $F_1'(1/t)$ have the Taylor expansions 
\[F_1(1/t) = 2 - \suml_{n=1}^\infty\frac{2t^{2n}}{4n^2-1},\quad F_1'(1/t) = \suml_{n=1}^\infty\frac{4nt^{2n+1}}{4n^2-1}.\]

\begin{lemma}\label{lem:F1_property} 
The function $F_1$ defined in \eqref{eqt:F1_definition} satisfies
\begin{enumerate}
\item $F_1(1/t) = 2-F_1(t)$, $F_1'(1/t) = t^2F_1'(t)$; 
\item $F_1\in C([0,+\infty))$, $F_1(0) = 0$, $F_1(1) = 1$, $\lim_{t\rightarrow+\infty}F_1(t)=2$, $\lim_{t\rightarrow0}F_1(t)/t = 0$;
\item $F_1'(0)=0$ and $F_1'(t)>0$ for $t> 0$.
\end{enumerate}
\end{lemma}

\begin{proof} Property ($1$) is straightforward to check. ($2$) follows from the Taylor expansion of $F_1(t)$ and property ($1$). ($3$) follows from the Taylor expansion of $F_1'(t)$ and property ($1$).
\end{proof}

\subsection{Special function $F_2$}\label{sec:F_2}

We define
\begin{equation}\label{eqt:F2_definition}
F_2(t):= \frac{3t^4 - 2t^2 - 1}{8t^3}\ln\left|\frac{t+1}{t-1}\right| + \frac{1}{4t^2} + \frac{7}{12}, \quad t\geq 0.
\end{equation}
The derivative of $G$ reads
\[F_2'(t) = \frac{3t^4 + 2t^2 + 3}{8t^4}\ln\left|\frac{t+1}{t-1}\right| - \frac{3t^2+3}{4t^3}.\]
For $t\in[0,1)$, $F_2(t)$ and $F_2'(t)$ have the Taylor expansions 
\[F_2(t) = \suml_{n=1}^{+\infty}\frac{4(n+1)t^{2n}}{(2n-1)(2n+1)(2n+3)},\quad F_2'(t) = \suml_{n=1}^{+\infty}\frac{8n(n+1)t^{2n-1}}{(2n-1)(2n+1)(2n+3)}.\]
For $t\in[0,1)$, $F_2(1/t)$ and $F_2'(1/t)$ have the Taylor expansions 
\[F_2(1/t) = \frac{4}{3}-\suml_{n=1}^{+\infty}\frac{4nt^{2n+2}}{(2n-1)(2n+1)(2n+3)},\quad F_2'(1/t) = \suml_{n=1}^{+\infty}\frac{8n(n+1)t^{2n+3}}{(2n-1)(2n+1)(2n+3)}.\]

\begin{lemma}\label{lem:F2_property} 
The function $F_2$ defined in \eqref{eqt:F2_definition} satisfies
\begin{enumerate}
\item $F_2'(1/t) = t^4F_2'(t)$;
\item $F_2\in C([0,+\infty))$, $F_2(0) = 0$, $F_2(1) = 5/6$, $\lim_{t\rightarrow+\infty}F_2(t)=4/3$, $\lim_{t\rightarrow0}F_2(t)/t=0$;
\item $F_2'(t)\geq 0$ for $t\geq 0$.
\item $(4t/3-tF_2(1/t))' = tF_1'(1/t)$ for $t\geq 0$.
\end{enumerate}
\end{lemma}

\begin{proof} Properties $(1)$ is straightforward to check. $(2)$ follows from the Taylor expansions of $F_2(t)$ and $F_2(1/t)$. $(3)$ follows from the Taylor expansion of $F_1'(t)$ and property $(1)$. $(4)$ can be checked straightforwardly by the definitions of $F_2(t)$ and $F_1(t)$.
\end{proof}

\subsection{The Hilbert transform}\label{sec:Hilbert_transform}

\begin{lemma}\label{lem:Hilbert_property}
For any suitable function $\om$ on $\mathbb{R}$,
\[\frac{\mtx{H}(\om)(x)-\mtx{H}(\om)(0)}{x} = \mtx{H}\left(\frac{\om - \om(0)}{x}\right)(x).\]
As a result, 
\[\frac{1}{\pi}\int_{\mathbb{R}}\frac{\mtx{H}(\om)(x)\cdot \om(x)}{x}\idiff x = \frac{1}{2}\om(0)^2 -\frac{1}{2}\big(\mtx{H}(\om)(0)\big)^2.\]
\end{lemma}

\begin{proof}
The first equation follows directly from the definition of the Hilbert transform on the real line. The second equation is derived from the first one as follows:
\begin{align*}
\frac{1}{\pi}\int_{\mathbb{R}}\frac{\mtx{H}(\om)\cdot \om}{x}\idiff x &= \frac{1}{\pi}\int_{\mathbb{R}}\frac{(\mtx{H}(\om)-\mtx{H}(\om)(0))}{x}\cdot \om\idiff x + \mtx{H}(\om)(0)\cdot \frac{1}{\pi}\int_{\mathbb{R}}\frac{\om}{x}\idiff x\\
&= \frac{1}{\pi}\int_{\mathbb{R}}\mtx{H}\left(\frac{\om-\om(0)}{x}\right)\cdot \om\idiff x - \big(\mtx{H}(\om)(0)\big)^2\\
&= -\frac{1}{\pi}\int_{\mathbb{R}}\frac{\om-\om(0)}{x}\cdot \mtx{H}(\om)\idiff x - \big(\mtx{H}(\om)(0)\big)^2\\
&= -\frac{1}{\pi}\int_{\mathbb{R}}\frac{\om\cdot \mtx{H}(\om)}{x}\idiff x + \om(0)\cdot \frac{1}{\pi}\int_{\mathbb{R}}\frac{\mtx{H}(\om)}{x}\idiff x - \big(\mtx{H}(\om)(0)\big)^2\\
&= -\frac{1}{\pi}\int_{\mathbb{R}}\frac{\mtx{H}(\om)\cdot \om}{x}\idiff x + \om(0)^2 - \big(\mtx{H}(\om)(0)\big)^2.
\end{align*}
Rearranging the equation above yields the desired result.
\end{proof}

\section{The Schauder fixed point theorem}

For the reader's convenience, we state below the Schauder fixed point theorem (see also, e.g., \cite[Theorem 2.2]{bonsall1962lectures} and \cite[Theorem 11.1]{gilbarg1977elliptic}) that we use in the proof of Theorem \ref{thm:existence_fixed_point} to conclude the existence of a fixed point of $\mtx{R}$ in $\mathbb{D}$.

\begin{fact}[Schauder fixed point theorem]\label{fact:Schauder} Let $\mathbb{K}$ be a nonempty compact convex set in a Banach space $\mathbb{B}$ and let $\mtx{F}$ be a continuous mapping of $\mathbb{K}$ into itself. Then $\mtx{F}$ has a fixed point in $\mathbb{K}$, that is, $\mtx{F}(\vct{x}) = \vct{x}$ for some $\vct{x}\in \mathbb{K}$.
\end{fact}

To be clear, the Banach space $\mathbb{V}$, the compact convex subset $\mathbb{D}$, and the nonlinear map $\mtx{R}$ in this paper play the roles of $\mathbb{B}$, $\mathbb{K}$, and $\mtx{F}$, respectively, in the Schauder fixed point theorem above.

\vspace{2mm}

\subsection*{Acknowledgement} The authors are supported by the National Key R\&D Program of China under the grant 2021YFA1001500.

\bibliographystyle{myalpha}
\newcommand{\etalchar}[1]{$^{#1}$}

\end{document}